\documentclass[a4paper,11pt]{article}
\usepackage[top=2cm, bottom=2cm, left=2cm, right=2cm]{geometry}

\usepackage{dsfont, minipage-marginpar}
\usepackage{amsmath, amsthm}
\usepackage{enumitem}

\usepackage[dvipsnames]{xcolor}
\usepackage[normalem]{ulem}

\usepackage{comment}

\numberwithin{equation}{section}
\newtheorem{thm}{Theorem}

\newtheorem{corollary}{Corollary}				
\newtheorem{proposition}{Proposition}
\newtheorem{definition}{Definition}
\newtheorem{assumption}{Assumption}
\newtheorem{remark}{Remark}

\newcommand{\change}[1]{#1}

\title{Kalikow decomposition for counting processes with stochastic intensity  \change{and application to simulation algorithms}} 

\author{Tien Cuong Phi\thanks{Universit\'e C\^ote d'Azur, CNRS, LJAD, France. Email: cuong.tienphi@gmail.com }, Eva L\"ocherbach\thanks{Universit\'e Paris 1 Panth\'eon-Sorbonne, Statistique, Analyse et Mod\'elisation Multidisciplinaire EA 4543 et FR FP2M 2036 CNRS, France. Email: eva.locherbach@univ-paris1.fr}, Patricia Reynaud-Bouret\thanks{Universit\'e C\^ote d'Azur, CNRS, LJAD, France. Email: Patricia.Reynaud-Bouret@univ-cotedazur.fr } }

\begin{document}

\maketitle

\begin{abstract}
We propose a new Kalikow decomposition 
for continuous time multivariate counting processes, on potentially infinite networks. We prove the existence of such a decomposition in various cases. This decomposition \change{allows us to derive simulation algorithms that hold either for stationary processes  with potentially infinite network but bounded intensities, or for processes with unbounded intensities in a finite network and with empty past before 0. The Kalikow decomposition is not unique and we discuss the choice of the decomposition in terms of algorithmic efficiency in certain cases.} We apply these methods on several examples:  linear Hawkes process, age dependent Hawkes process, exponential Hawkes process, Galves-Löcherbach process.
\end{abstract}

{\bf Keywords:} Kalikow decomposition; counting process; Hawkes process; Perfect Simulation; Simulation algorithms.

{\bf MSC 2010 subject classification:} 60G55, 60K35 

\maketitle

\section{Introduction} \label{sec:1}

Multivariate point (or counting) processes on networks have been used to model a large variety of situations: social networks \cite{HW16}, financial prices \cite{BHDM}, genomics \cite{RS10}, etc. One of the most complex network models comes from neuroscience where the number of nodes can be as large as billions \cite{PMR, MMR, SMMPR}. Several counting process models have been used to model such large networks: Hawkes processes \cite{Hawkes71a, Hawkes71b}, Galves-L\"ocherbach models \cite{GL2013} etc. \change{ The simulation of such large and potentially infinite networks is of fundamental importance in computational neuroscience \cite{MMR, PMR}. From a more mathematical point of view, } the existence of such processes in stationary regime and within a potentially infinite network draws \change{also} a lot of interest (see \cite{GL2013, ORB, HL} in discrete or continuous time).

Kalikow decompositions \cite{Kal90} have been introduced and mainly used in discrete time. Such a decomposition provides a decomposition of the  transition probabilities into a mixture of more elementary transitions. The whole idea is that even if the process is complex (infinite memory, infinite network), the elementary transitions look only at what happens in a finite neighborhood in time and space. Once the decomposition is proved for a given process, this can be used to write \change{algorithms to simulate the process. Indeed,  by the Kalikow decomposition, the process can be decomposed in random elementary transitions that do not need to access to the whole information to move forward in time.  This is useful in two ways. First one can simulate the points appearing for  a given node in the network, without needing the whole past history or  even the whole network. This leads to} Perfect Simulation algorithms \cite{GL2013, HL, ORB}. Here the word ``perfect" refers to the fact that it is possible in finite time to simulate what happens on one or on a finite number of the nodes of the potentially infinite network in a stationary regime. \change{Secondly, this decomposition can drastically reduce the time complexity of the simulation algorithms, because we do not need to store and/or compute at each step what happens in the whole network to proceed.} 
\change{Up to our knowledge,} all existing papers referring to Kalikow decomposition 
are theoretical \change{and focus on the first aim: indeed if one can prove that such a Perfect simulation algorithm ends after a  finite number of steps, 
it means at the same time that the process exists in a stationary regime. This is of tremendous theoretical importance when dealing with  infinite networks} \cite{HL, GL2013}. 

In the present paper, we propose to go from discrete to continuous time. Therefore, we 
decompose conditional intensities rather than transition probabilities. This leads to serious difficulties that usually prevent a more practical application of the 
\change{simulation algorithms}. Indeed, up to our knowledge, the only work dealing with continuous time counting processes is the one by Hodara and L\"ocherbach \cite{HL}. Their decomposition is constructed under the assumption that there is a dominating Poisson process on each of the nodes, from which the points of the processes under interest can be thinned by rejection sampling (see also \cite{Ogata} for another use of thinning in simulation of counting processes). To prove the existence of a Kalikow decomposition and go back to a more classical discrete time setting, the authors need to freeze the dominating Poisson process, leading to a mixture, in the Kalikow decomposition, that depends on the realization of the dominating Poisson process. Such a mixture is not accessible in practice, and this prevents the use of their Perfect Simulation algorithm for more concrete purposes than mere existence.

More recently, in a previous computational article \cite{PMR}, we have used another type of Kalikow decomposition, which does not depend on the dominating Poisson process. This leads to a Perfect Simulation algorithm, which can be used  as a concrete way for Computational Neuroscience to simulate neuronal networks as an open physical system, where we do not need to simulate the whole network to simulate what happens in a small part of it \cite{PMR}. \change{However, this approach was mainly done in direction of computer science and the definition of the Kalikow decomposition was not sufficiently broad to encompass the classical counting process examples such as classical linear Hawkes processes.}

In the present work, we want to go further, by proposing a Kalikow decomposition in a general frame, which does not assume the existence of a dominating Poisson process at all. We also prove (and this is not done in \cite{PMR})  that such a decomposition exists for various interesting examples, even if it is not unique. Finally we propose \change{two algorithms. The first one is essentially the one that is proposed in \cite{PMR}, except that it is now proved to work for the more general definition of the Kalikow decomposition that we have introduced here. This is a Perfect Simulation algorithm in the sense that it can simulate the point process of a given node in a potentially infinite network in a stationary regime (it moves backward to create the past that is needed to obtain a stationary process). To do so, we need to assume the existence of a dominating Poisson measure. The second algorithm only moves forward from an empty past (meaning no  points before 0) in a finite network and therefore simulates outside the stationary regime. In this case, the intensity does not need to be bounded, as in the more classical Ogata's algorithm \cite{Ogata}.}

The paper is organized as follows. In Section \ref{sec:2}, we introduce the basic notation and give the precise definition of a Kalikow decomposition. In Section \ref{sec:3}, we present \change{a very general method}  to obtain a Kalikow decomposition for a counting process having stochastic intensity. \change{We decline this method on various examples including the linear Hawkes process \cite{Hawkes71a, Hawkes71b}, the age dependent Hawkes process \cite{RDL}, more general non-linear Hawkes processes with analytic rate function and Galves-L\"ocherbach models \cite{GL2013}.}

Finally, in Section \ref{sec:4}, we present \change{the two algorithms} 
based on the Kalikow decomposition written in Section \ref{sec:3}, and we discuss the efficiency of the \change{Perfect simulation} algorithm with respect to the Kalikow decomposition.

\section{Notation and Kalikow decomposition} \label{sec:2}
\subsection{Notation and Definition}
We start this section by recalling the definition of simple locally finite counting processes and stochastic intensities. We refer the reader to \cite{Bre81} and \cite{DVJ} for more complete statements.

Let $\textbf{I}$ be a countable index set. \change{We start by introducing a canonical path space for the sequence of points of a counting process having interacting components indexed by $ i \in \bf{I}. $ This space is given by
$$\mathcal{X}_\change{\infty}= \left\{  (\{ t^i_n \}_{n \in \mathds{Z}} )_{i \in \textbf{I}} \mbox{ s.t. } \forall i\neq j \in {\bf I}, n, m  \in \mathds{Z}, \left\{ \begin{array}{ll}
t^i_n \in [- \infty, \infty ] & \\
\lim_{ n \to \pm \infty} t^i_n = \pm \infty &\\
 t^i_n < t^i_{n+1} \mbox{ and }t^i_n \neq t^j_m &   \mbox{if they are not infinite}\\
\end{array}
\right.
\right\} .
$$}
\change{Notice that we allow for the choices $t^i_n = \pm \infty $ such that elements of $ \mathcal{X}_\infty$ may have only a finite number of (finite) points  before time $0$ or only a finite number of  (finite) points after time $0, $ or both. In the sequel, whenever we speak of the points of a counting process, we implicitly mean finite points. } 

\change{We then introduce, for any $ t \in \mathds{R} $, $ i \in \bf{I}$ and $x=(\{ t^i_n \}_{n \in \mathds{Z}} )_{i \in \textbf{I}}$,
$$ Z^i_t (x) = \sum_{ n \geq 1} {\bf 1}_{ t_n^i \le t } , \; \mbox{ if }  t \geq 0, \mbox{ and } 
 Z^i_t (x) = - \sum_{ n \leq 0 } {\bf 1}_{ t < t_n^i },  \;  \mbox{ if } t \le 0 ,$$  
and we write for short $ Z=(Z^i)_{ i \in \bf{I} }$ for the associated collection of counting processes, indexed by $ i \in \bf{I}.$ Note that  $Z \in D ( \mathds{R}, \mathds{Z})^{\bf I } $,  where  $D ( \mathds{R}, \mathds{Z})$ is the space of non decreasing c\`adl\`ag piecewise constant functions, see \cite{JS}.} 
We denote $\mathcal{X}_{t} $ the canonical path space of points of $Z$ before time $t,$ given by
\begin{equation*}
\mathcal{X}_{t} = \mathcal{X}_\change{\infty} \cap (-\infty ,t)^{\bf I},
\end{equation*}
and \change{we identify $ (Z_s)_{s <  t } $  with the 
 past configuration $X_t \in \mathcal{X}_t$}   defined by 
\begin{equation}\label{eq:pastconfig}
 X_t ( x) =  (\{ t^i_n \}_{ t^i_n < t } )_{i \in \textbf{I}} .
\end{equation} 
We consider $(\mathcal{F}_{t})_{t \in \mathds{R}},$ the past filtration of the process $Z = (Z^i)_{i \in \textbf{I}},$ defined by
$$
\mathcal{F}_{t} = \sigma (Z^i_s, i \in \textbf{I}, s \leq t) .
$$
Moreover, for any $x =  \left(\{t^i_n\}_{n \in \mathds{Z_{-}}} \right)_{i \in \textbf{I}} \in \mathcal{X}$ and any $i \in \textbf{I}$, we denote the point measure associated to index $i$ by
\begin{equation*}
dx^i_s = \sum_{m \in \mathds{Z}^{-}} \delta_{t^i_m}(ds)
\end{equation*}
\change{
which is an element of the space $ \mathcal{N}$ of locally finite point measures on $\mathds{R}_-,$ endowed with the topology of vague convergence,
and we endow $ \mathcal{X} $ with the metric induced by the product metric on $ \mathcal{N}^{\textbf{I}}.$} 
Finally, throughout this article, without further mentioning, the integral $\int_{a}^{b}$  stands for $\int_{[a,b)}$ with $a,b \in \mathds{R}$ and $Z^i([a,b))$ ($Z^i ( (a, b ]), $ respectively) stands for the number of points in the configuration with index $i$ in $[a,b)$ (in $(a, b ] , $ respectively).

\change{In the present article, we are only interested by time homogeneous counting processes, that is, informally, processes that at each time  $t$ depend in the same way on the past configuration $x_t$. To be more rigorous, let us define,
for each $x_t = \left(\{t^i_n\}_{t^i_n < t } \right)_{i \in \textbf{I}} \in \mathcal{X}_{t}$, $$x_t^{\leftarrow t}:= \left( \{ t^i_n - t \}_{n} \right)_{i},$$
which is the shifted configuration at time $0$. The generic space for such a past configuration $x_t^{\leftarrow t}$ that is rooted at time 0 is denoted $\mathcal{X}:=\mathcal{X}_{0}  $. 
}

Under suitable assumptions, the evolution of the counting process $Z = (Z^i)_{i \in \textbf{I}}$ with respect to $(\mathcal{F}_{t})_{t \in \mathds{R}}$ is fully characterized by its stochastic intensity which depends on the past configuration, see Proposition 7.2.IV of \cite{DVJ}. Hence, in this paper, for any $x_t \in \mathcal{X}_t$, given that the past before time $t$ is $x_t$, we denote by $\phi^i_{t}(x_t)$ the corresponding stochastic intensity of the process $Z^i$ at time $t$ for any $i \in \textbf{I}$. 
More precisely, for any  $x_t \in \mathcal{X}_t$, we have
$$
\mathds{P}\left(Z^i \, \text{has a jump in} \, [t, t+dt) \mid \text{past before time} \, t = x_t \right) = \phi^i_{t}(x_t) dt .
$$

\begin{definition} \label{time homogen}
For a given $i \in \textbf{I}$, a counting process $Z^i$ with stochastic intensity $(\phi^i_{t}(x_t))_{t \in \mathds{R}}$ is said to be time homogeneous \change{if there exists a measurable function $\phi^i :\mathcal{X} \to \mathds{R}_+ $, called the {\bf generic intensity,} such that}
\begin{equation*}
\phi^i_t(x_t) = \phi^{i}(x_t^{\leftarrow t})
\end{equation*}
for all $t \in \mathds{R}$ and $x_t \in \mathcal{X}_t$.
\end{definition}

\change{Note that a process that is time homogeneous is not necessarily stationary. For instance, exponential Hawkes Processes \cite{Cars} (see Section \ref{sec_nonlin} for more details) starting with empty past before time $0$ (that is empty past history before time $0$, or in other words, no points before $0$) may explode in finite time, but they are still time homogeneous in the sense of the previous definition. One can also think of simple linear Hawkes processes with empty past before time $0$ in the supercritical regime, that is,  for which the interaction function has $L^1-$norm larger than $1$ and thus produces an exponentially growing number of points as time increases \cite{BHDM}. Therefore, if the process of interest is stationary, one can think of $\mathcal{X}$ as $\mathcal{X}_0$ the set of configurations at time 0. However if this is not the case, $\mathcal{X}$ has to be thought as just a generic space whose configurations are cut at time 0 (that is no points exists after 0) and which is the set on which $\phi^i$ is defined.}

\change{To define the Kalikow decomposition, we need to define neighborhoods and cylindrical functions on neighborhoods. A neighborhood $v$ is a Borel subset of $\textbf{I} \times (-\infty, 0)$. This neighborhood is said to be finite if there exists a finite subset $J \subset \textbf{I}$ and a finite interval $[a, b]$ such that:
$$
v \subset J \times [a,b].
$$

\begin{definition}
For any neighborhood $v$ and $x,y\in \mathcal{X}$, we say $x\overset{v}{=} y$ whenever $x= y$ in $v$. This means that, for all $i \in \textbf{I} , n \in \mathds{Z}$, such that $t^i_n \in x$ and $(i,t^i_n) \in v$, we have $t^i_n \in y$ and vice-versa. 

A real valued function $f$ is said cylindrical in $v$ if $f(x) = f(y)$ for any $x\overset{v}{=} y$, and we usually stress the dependence in $v$ by denoting $f_{v}(x)$.
\end{definition}

A family of  neighborhoods ${\bf V}$ is a countable collection  of finite neighborhoods $v$. It usually includes the empty set $\emptyset$. One might have a different family of neighborhood for each $i$, even if in several examples, the same family works for all $i$.}

\change{
\begin{definition} \label{def:article-2/uncoditionalKalikow}
A time homogeneous counting process $(Z^i)_{i \in {\bf I}}$ of generic intensity $\phi^i$  admits the Kalikow decomposition with respect to (w.r.t.) the collection of neighborhood families $({\bf V}^i)_{i \in {\bf I}}$ and a given subspace $\mathcal{Y}$ of $\mathcal{X}$ if, for any $i \in {\bf I}$ and any  $v \in \bf{V}^i$ there exists a cylindrical function $\phi^i_{v}(.)$ on $v$ taking values in $\mathds{R}_{+}$ and a probability $\lambda^i(.)$  on ${\bf V}^i$ such that

\begin{equation} \label{Kalikow_decomposition}
\forall x \in \mathcal{X} \cap \mathcal{Y} \qquad  \phi^i(x) = \sum\limits_{v \in \bf{V}^i}\lambda^i(v)\phi^i_{v}(x).
\end{equation}
\end{definition}}

\begin{remark}
Note that the \change{probability} $\lambda^i(.)$ in Definition \ref{def:article-2/uncoditionalKalikow} is a deterministic function, that is why this decomposition is unconditional, whereas in \cite{HL}, $\lambda^i(.)$ was depending on the dominating Poisson processes (see the discussion in the introduction). Secondly, we do not restrict ourself to a bounded intensity, \change{and we do not force all the $\phi^i_v$ to be bounded with the same bound}, which is a notable improvement compared to \cite{PMR}.
\end{remark}

\subsection{About the subspace $\mathcal{Y}$} \label{subspace}

\change{
A Kalikow decomposition does not exist for all intensities and all subspaces $\mathcal{Y},$ and we stress the fact that it depends on the choice of $\mathcal{Y}.$
The role of $\mathcal{Y}$ is to make the Kalikow decomposition achievable. There are many possible choices for such a subspace, depending on the model under consideration and the precise form of the intensity.  In this paper, we will discuss two main examples. The first example is the choice  $\mathcal{Y} = \mathcal{X}^{> \delta}$ which is the subspace of $\mathcal{X}$ where the distance between any two consecutive possible points is greater than $\delta,$ that is, 
\begin{equation} \label{Xdelta}
\mathcal{X}^{> \delta} = \{ x= (\{t^i_n\}_{n \in \mathds{Z}^{-}})_{i \in \textbf{I}} \in \mathcal{X} \quad \text{such that} \quad \forall n,i  \quad t^i_{n+1} - t^i_{n} >\delta  \}.
\end{equation}
Such a choice is convenient for counting processes with a hard exclusion role where by definition of the intensity, any two consecutive points need to be at a distance at least equal to  $\delta,$ see Section \ref{hard AHP} below where we discuss the example of Age dependent Hawkes processes with hard refractory period. 

In the case when $ \phi^i $ is continuous for each $i $ and we want to simulate the process starting from the empty past before time $0,$ during some finite time interval $ [0, T ] $ and up to some activity level $ K > 0, $ another possible choice of a subspace $\mathcal{Y}$ that we consider is 
$$ \mathcal{Y} = \mathcal{X}^{T, K } = \{x \in \mathcal{X}, \, \forall i \in {\bf{I}}, Z^i_T (x)  \le K \}.$$
By continuity of $ \phi^i, $   it is clear that the intensities are bounded on $ \mathcal{X}^{T, K }  .$ In this case, we will be able to simulate the process, using the Kalikow decomposition,  up to the first exit time of $\mathcal{X}^{T, K }, $ see Section \ref{Sec:forward} below.}

\subsection{Representation, thinning and simulation \label{thinning}} 

\change{Locally finite simple counting processes with intensity $\phi^i_t(x_t)$ can always be represented as thinning of a bivariate Poisson measure. More precisely, if $(\pi^i)_{i\in {\bf I}}$ are independent Poisson random measures on $\mathds{R}\times\mathds{R}_+$ with intensity 1, the point measures defined by
$$dX^i_t= \int_{u\geq 0} {\bf 1}_{u\leq \phi^i_t(X_t)} \pi^i(dt,du)$$
define points of  counting processes $(Z^i)_{i\in {\bf I}}$ having intensity $\phi^i_t(x_t)$ at time $t,$ given that the past before time $t$ is $x_t,$ see e.g. Lemma 3 and 4 of \cite{BM96} and Chapter 14 of \cite{Jacod}. 

This representation has been used for a long time to simulate  processes forward in time, see e.g. \cite{Ogata} and \cite{Lewis}. More precisely, consider for the moment the easy case where we have empty past before time $0$ and intensities which are bounded for any $ i $ by a fixed constant $\Gamma^i > 0 .$ The simulation of the Poisson random measure $\pi^i$ then consists in a homogeneous Poisson process, $N^i$, in time, having intensity $\Gamma^i$ (that can also be built by a succession of independent exponential jumps of parameter $\Gamma^i$), and then to attach to each point $T$ of this process, independently of anything else, independent marks $U_T$ which are uniformly distributed on $[0,\Gamma^i]$. The couples $(T,U_T)$ for $T\in N^i$ form  the bivariate Poisson random measure  $\pi^i$ in the band of height $\Gamma^i$.

The classical thinning algorithm -- in case of intensities which are bounded by $ \Gamma^i$-- then consists in saying that the points of $N^i$ are accepted if $U_T\leq  \phi^i_T(X_T)$  and that these accepted points correspond to a counting process of intensity $\phi^i_t(X_t)$. Equivalently, one can attach independent uniform marks $\mathcal{U}_T$ on $[0,1]$ and say that we accept $T$ if $\mathcal{U}_T \leq \phi^i_T(X_T)/\Gamma^i$, or one can even just say that we accept a point  $T$ in $N^i$ with probability $\phi^i_T(X_T)/\Gamma^i$.

The Kalikow decomposition allows to go one step further thanks to the following result.}

\begin{proposition}\label{thinn_kali}
\change{Let  $({\bf V}^i)_{i\in {\bf I}}$ be a collection of families of finite neighborhoods and for any $ i \in \bf{I}, $  $(\lambda^i(v))_{v\in {\bf V}^i}$ be probabilities on ${\bf V}^i$ and $\phi^i_v$ be cylindrical functions on $v$ for each $v\in {\bf V}^i.$ Let moreover $(\Pi^i)_{i\in {\bf I}}$ be  independent Poisson measures on $\mathds{R}\times\mathds{R}_+\times {\bf V}^i$ with intensity measure $dt~du~ \lambda^i(dv)$.

Then for every $i \in {\bf I}$, 
$$dZ^i_t= \int_{u\geq 0, v\in {\bf V}^i} {\bf 1}_{u\leq \phi^i_{v}(X_t^{\leftarrow t})} \Pi^i(dt,du,dv)$$
defines a time homogeneous counting process $Z^i$ having generic intensity given by
\begin{equation}\label{lasomme}
\phi^i= \sum\limits_{v \in \bf{V}^i}\lambda^i(v)\phi^i_{v}.
\end{equation}}
\end{proposition}

 \change{From a more algorithmic point of view, the construction of $\Pi^i$ is equivalent to attaching to each atom $(T,U_T)$ of $\pi^i$ (the bivariate Poisson random measure) a mark $V_T$, independently of everything else and distributed according to $\lambda^i$.
The above theoretical result can be interpreted in the following way. It is sufficient to draw at random the neighborhood $V_T$ for each point $T$ according to $\lambda^i$ and to do as if the intensity was just $\phi^i_{V_T}(X_t^{\leftarrow t})$ instead of having to compute the full sum in \eqref{lasomme}.

From a computational point of view, the main interest of this is to diminish drastically the number of computations to be done, replacing the summation by a random selection. From a theoretical point of view, the above representation enables us to define reset events at which the process forgets its past, at least in a local way. These resets take place whenever, for instance, the empty set is picked as a neighborhood, which happens with probability $\lambda^i ( \emptyset ), $ for a given $ i \in \bf I.$ Indeed, this means that at this time $t$, the history of the $i-$th component is reset and becomes independent of the past. In a nutshell, the whole theoretical interest of using the Kalikow decomposition for Perfect simulation is therefore to prove that such resets happen often enough and for sufficiently many coordinates $i, $ and that one can simulate  the distribution between resets, without needing to simulate outside of this zone. The interest of this strategy is twofold. First of all, if this procedure works, this shows that the stationary distribution exists and is unique (see Theorem \ref{theo:stationary} below). Moreover, it also allows to perfectly simulate from this stationary distribution. } 

\change{
The above representation needs a tridimensional Poisson random measure. As we have seen above, if it is easy to simulate such a measure within a band, we cannot simulate it without having an upper bound on the intensity. That is why we distinguish two cases for the algorithms. Here is a brief informal overview of them.
\begin{itemize}
\item If we seek for a stationary distribution, we need to ``propose" a first point before thinning it in a ``Kalikow'' way. To do so, we need a fixed upper bound, say $\Gamma^i,$ that holds for all times, for a given coordinate $i.$ Then we will be able to recursively go back in time as follows:  (i) pick the random neighborhood (ii) simulate the points in the  neighborhood according to a Poisson process of intensity $\Gamma^j$ if the neighborhood is on node $j$ (iii) search again for the neighborhoods of the points that we just simulated and go on. In this backward step, we create a clan of ancestors for the first point. Under some conditions, this recursion ends in finite time, because either the empty neighborhood is picked or because the simulation of the Poisson process inside the neighborhood is empty. Hence the status of the first point, even inside the stationary distribution,  only depends on a finite set of points that we have been able to create. It remains to accept or reject, in a forward movement, all these points  according to the rule $U_T\leq \phi^i_{V_T}(x_T^{\leftarrow T}).$ See Section \ref{sec:41} below. 
\item If there is no bound, which is typically the case of the linear or exponential Hawkes process, one cannot ``propose" a first point in a stationary manner. However, if we start with an empty past before $0,$ the intensity at time $0$ is usually bounded and we can ``propose" a first point. The main strategy is now to update this upper bound as time goes by and to change the size of the steps we are making, one point after the other. This algorithm can only move in forward time and one needs to know all the network to make it move forward. Hence this approach does only work for finite networks with known past (empty past before time $0,$ typically). See Section \ref{Sec:forward} below.
\end{itemize}

Note that in particular for linear (multivariate) Hawkes processes, there is another way to do Perfect Simulation, relying on the cluster representation of these processes, see \cite{MR}, \cite{Chen} and \cite{ChenWang}.  However the approach we propose here is much more general, because we do not need this cluster representation, but rely on the Kalikow decomposition which exists for a broader class of processes (for instance non-linear Hawkes processes with hard refractory period and Lipschitz rate function, see Section \ref{hard AHP} below).}

The above proposition holds of course as long as  $\sum\limits_{v \in \bf{V}^i}\lambda^i(v)\phi^i_{v}$ exists and it furnishes a way to simulate a process with intensity $\phi^i$. In the sequel, we will see how we combine it with a proper definition of $\mathcal{Y}$ to make the simulation algorithms work in practice.

\change{\begin{proof}[Proof of Proposition \ref{thinn_kali}]
To avoid confusion, we consider two filtrations. We write  $(\mathcal{F}_t)_{t\in \mathds{R}}$ for the canonical filtration of $(\Pi^i)_{i\in {\bf I}}$, containing the filtration $(\mathcal{F}^Z_t)_{t\in \mathds{R}}$ which corresponds to the one given by the counting processes $(Z^i)_{i\in {\bf I}}$.

To prove that $(Z^i)_{i\in {\bf I}}$ has the correct intensity, let us look at $\mathds{E}\left(Z^i ( (a, b ] |  \mathcal{F}^Z_a \right),$
for any $ a < b .$ We have that 
$$\mathds{E} \left(Z^i ( (a, b ] |  \mathcal{F}^Z_a \right) =\mathds{E}\left(\int_{s \in (a,b],u\geq 0, v\in {\bf V}^i} {\bf 1}_{u\leq \phi^i_{v}(X_s^{\leftarrow s})} \Pi^i(ds,du,dv)|  \mathcal{F}^Z_a \right).$$
The integral in $v$ is independent from the rest, so we can integrate it and replace it by its corresponding intensity measure, which leads to
$$\mathds{E} \left(Z^i ( (a, b ] |  \mathcal{F}^Z_a \right) =\mathds{E}\left(\sum_{v\in {\bf V}^i} \lambda^i(v) \int_{s \in (a,b],u\geq 0} {\bf 1}_{u\leq \phi^i_{v}(X_s^{\leftarrow s})} \pi^i(ds,du)|  \mathcal{F}^Z_a  \right),$$
where $\pi^i$ is the bivariate Poisson random measure of rate 1. 
Therefore we get
$$\mathds{E} \left(Z^i ( (a, b ] |  \mathcal{F}^Z_a \right)=\sum_{v\in {\bf V}^i} \lambda^i(v)  \mathds{E}\left( \int_{s \in(a,b]}  \phi^i_{v}(X_s^{\leftarrow s})ds |  \mathcal{F}^Z_a \right),$$
that is 
$$\mathds{E} \left(Z^i ( (a, b ] |  \mathcal{F}^Z_a \right)= \mathds{E}\left( \int_{s \in(a,b]} \phi^i_{s}(X_s)ds |  \mathcal{F}^Z_a \right) ,$$
which concludes the proof.
\end{proof}}

\section{\change{How to compute the Kalikow decomposition in various cases}} \label{sec:3}
\change{To obtain the Kalikow decomposition, we try to rewrite the intensity as a convergent sum of cylindrical functions over an adequate family of neighborhoods.
To define properly such a family of neighborhoods, let us start by introducing the minimal information that is needed to compute the intensity.}

\change{Let us consider a time homogeneous counting process $Z^i$ having intensity $\phi^i_t(x_t) = \phi^{i}(x_t^{\leftarrow t}).$  Let  $\mathcal{V}^i \subset {\bf I}\times (-\infty,0)$ be minimal such that $\phi^i$ is cylindrical on $\mathcal{V}^{ i}.$ We interpret $\mathcal{V}^i$ as the support of dependance for $\phi^i$. }

\begin{definition}
\change{ A coherent family of finite neighborhoods ${\bf V}^i$  for $Z^i$  is such that
\begin{equation} \label{Neighborhood Family}
\mathcal{V}^i \subset \bigcup_{v\in {\bf V}^i} v.
\end{equation}}
\end{definition}
\change{
\begin{proposition}\label{prop:article-2/MethodToWriteKal}
Let  $Z=(Z^i)_{i\in {\bf I}}$ be a time homogeneous counting process having intensity  $\phi^i_t(x_t) = \phi^{i}(x_t^{\leftarrow t})$ at time $t, $ let ${\bf V}^i$ be an associated coherent family of neighborhoods and $\mathcal{Y} \subset \mathcal{X}$ a subspace of $ \mathcal{X} .$  If there exists a family of non-negative functions $\Delta^i_v(x)$ which are cylindrical on $v$ for each $v\in{\bf V}^i$ and ${i\in{\bf I}}$ such that
\begin{equation}\label{eq:convergent}
 \forall x \in \mathcal{X} \cap \mathcal{Y} , \; \sum_{v \in {\bf V}^i} \Delta^i_{v}(x) \mbox{ is convergent  and } \; \phi^{i}(x) = \sum_{v \in {\bf V}^i} \Delta^i_{v}(x),
\end{equation} 
then
\begin{enumerate}
\item whatever the  choice of the probabilities $\lambda^i$ on ${\bf V}^i$, such that 
$$\lambda^i(v)=0\quad \mbox{only if} \quad \sup_{x\in \mathcal{X} \cap \mathcal{Y}}  \Delta^i_{v}(x)=0 ,$$ 
then $(Z^i)_{i \in {\bf I}}$  admits the Kalikow decomposition with respect to the collection of neighborhood families $({\bf V}^i)_{i \in {\bf I}}$ and the subspace $\mathcal{Y},$ with weights given by  $\lambda^i(v)$ and cylindrical functions $\phi^i_v,$ where 
$$\phi^i_v (x)=  \Delta^i_{v}(x)/\lambda^i(v),$$ 
with the convention $0/0=0$.
\item if for all $i\in {\bf I}$ and $v \in {\bf V}^i$, there exist nonnegative deterministic constants $\Gamma^i_v$ such that 
$$ \sup_{x\in \mathcal{X} \cap \mathcal{Y}}  \Delta^i_{v}(x) \leq  \Gamma^i_v <\infty$$
and $\Gamma^i=\sum_{v\in {\bf V}^i} \Gamma^i_v\neq 0$ is finite,
then one can in particular choose
\begin{equation}\label{eq:choicelambda}
  \forall i\in {\bf I}, \forall v \in {\bf V}^i, \quad \lambda^i(v)=   \frac{\Gamma^i_v}{\Gamma^i} \quad \mbox{and} \quad \phi^i_v (x)=  \frac{\Gamma^i}{\Gamma^i_v} \Delta^i_{v}(x).
\end{equation}  
In this case, all functions $\phi^i_v$ and $\phi^i$ are  upperbounded by $\Gamma^i.$

\end{enumerate}
\end{proposition}}

\begin{proof}[Proof]
\change{The first point is obvious. Note that for the second one, since $\sum_{v\in {\bf V}^i}  \Gamma^i_v =\Gamma^i$, the choice $\lambda^i(v)=   \Gamma^i_v/\Gamma^i$ defines a probability.}

\end{proof}

\change{The weights $\lambda^i (v) $ in the Kalikow decomposition are not unique, and this even in the bounded case. Indeed, $\Gamma^i_v$ can always be chosen much bigger than $ \sup_{x\in \mathcal{X} \cap \mathcal{Y}}  \Delta^i_{v}(x)$. If we order the neighborhoods from the most simple (the empty set) to the most complex (by size, or by distance to $(i,0)$), this means that it is always possible to choose large weights $\lambda^i(v)$  on complex neighborhoods, but not that easy to choose small weights on them. However, we would like  the backward steps of the Perfect simulation algorithm to end after a  small finite number of iterations. Hence we typically  need larger weights on the empty set and on the less complex neighborhoods such that less computations are done or less memory is stored in our algorithms. 
This means that one can try to optimize the weights to  minimize the complexity of the algorithms that we develop (see Section \ref{sec:choice}). 

Note also that the above result leads directly to a very generic statement for continuous generic intensities.

\begin{corollary}
Let  $Z=(Z^i)_{i\in {\bf I}}$ be a time homogeneous counting process with generic intensity  $\phi^i$, let ${\bf V}^i = \{ v^i_k, k \in \mathds{N}\} $ be an associated coherent  family of neighborhoods which is increasing, that is, $ v^i_k \subset v^i_{k+1} $ for all $k, $  and let  $\mathcal{Y}$ be a subspace of $ \mathcal{X} ,$ such that for any $ i \in \bf{I}, $ $ \phi^i $ is strongly continuous on $ \mathcal{Y} ,$ that is,  
\begin{equation}\label{eq:strongcont}
\sup_{ i \in \bf{I} } \sup_{ x, y \in \mathcal{Y} : x\overset{v^i_k}{=} y} | \phi^i( x) - \phi^i(y) | \to 0 
\end{equation}
as $ k \to \infty .$ 

Then Item 1. of Proposition \ref{prop:article-2/MethodToWriteKal} holds with the choice 
$$\left\{
\begin{array}{ll}
\Delta^i_{v^i_0} (x)=\inf \{ \phi^i (y) : y \in \mathcal{Y}, \; y\overset{v^i_0}{=} x \}  & 
\\
\Delta^i_{v^i_k} (x) = \inf \{ \phi^i (y) : y \in \mathcal{Y}, \; y\overset{v^i_k}{=} x \} -  \inf \{ \phi^i (y) : y \in \mathcal{Y}, \; y\overset{v^i_{k-1}}{=} x \}  &\mbox{for} \quad k >0 . \\
\end{array}
\right. 
$$ 
\end{corollary}
}

\change{
Let us now see how this method can be implemented on various more specific examples of processes and various examples of neighborhood families.
}

\subsection{Linear Hawkes process.}
In the following, we consider a linear Hawkes process \cite{Hawkes71a, Hawkes71b}, for a finite number of interacting components, that is, $ \bf I$ finite. In this framework, for any $x \in \mathcal{X}$,
\begin{equation}\label{def_lin-Haw}
\phi^{i}(x)= \mu^i + \sum_{j \in \textbf{I}} \int_{-\infty}^{0} h^i_j(-s) dx^j_s ,
\end{equation}
where \change{the non negative interaction functions } $h^i_j(.)$ measure the local dependence of process $Z^i$ on $Z^j$ and the \change{non negative parameters $\mu^i$ } refer to the spontaneous rate of process $Z^i$. \change{The classical assumption to have a stationary Hawkes process is to assume that the spectral radius  of $(\int_0^{\infty} h^i_{j}(s) ds)_{i,j \in {\bf I}}$ is strictly smaller than 1. We do not need to make this assumption here, we just assume that for any $ i, j , $ $h^i_j \in L^1_{loc}.$ In this case we have
\begin{equation}\label{Vrondi}
\mathcal{V}^i= \bigcup_{j\in {\bf I}: h^j_i \neq 0} \{j\}\times Supp(h^i_j) .
\end{equation}
Cutting the support of $h^i_j$ into small pieces of length $\epsilon,$ for some fixed $ \epsilon > 0,$ we are led to consider atomic neighborhoods of the type 
\begin{equation}\label{atomic}
w_{j,n}=\{j\}\times[-n\epsilon, -(n-1)\epsilon)
\end{equation}
which are supported by one single neuron within a small interval of length $\epsilon$.
The most generic  family of neighborhoods that we can consider in this sense is
 $${\bf V}_{atom}=\{\emptyset\} \cup \{w_{j,n} :  j\in {\bf I}, n \in \mathds{N}^*\}$$ 
 that we choose for all $i$. For this family, one can prove the following result, as a straightforward corollary of Proposition \ref{prop:article-2/MethodToWriteKal}.

\begin{corollary}\label{linHaw_kali}
The multivariate linear Hawkes process defined by \eqref{def_lin-Haw} admits the Kalikow decomposition given by Proposition \ref{prop:article-2/MethodToWriteKal} Item 1., with respect to  the collection of neighborhood families $({\bf V}^i)_{i \in {\bf I}}$ where for all $i$, ${\bf V}^i={\bf V}_{atom},$ and with respect to the subspace $\mathcal{Y}=\{x \in \mathcal{X} : \sum_{i \in \bf I}  \phi^i(x)<\infty\},$ with 
\begin{itemize}
\item if $v=w_{j,n}$,
$$\Delta^i_{w_{j,n}}(x):= \int_{-n \epsilon}^{- n \epsilon + \epsilon}h^i_j(-s)dx^{j}_s$$
\item if $v=\emptyset$,
$$\Delta^i_\emptyset=\mu^i.$$
\end{itemize}
\end{corollary}
Note in particular that the linear Hawkes process has an intensity which is not bounded by a fixed constant, so we cannot use the second part of Proposition \ref{prop:article-2/MethodToWriteKal}. In particular we stress that we will not be able to use the above decomposition for the purpose of Perfect Simulation of the stationary version of the linear Hawkes process (which does not necessarily exist under our minimal assumptions). However, we will be able to use the decomposition for an efficient Sampling Algorithm of the process in its non-stationary regime, starting from the empty past before time $0,$ see Section \ref{Sec:forward} below. 

\begin{remark}
The choice of $ \mathcal{Y} =\{x \in \mathcal{X} : \sum_{i \in \bf I}  \phi^i(x)<\infty\}$ is the minimal choice to ensure that \eqref{eq:convergent} holds, by monotone convergence. By classical results on linear Hawkes processes, it is well known that, starting from the empty past at time $0,$ $Z_t = (Z_t^i)_{i \in \bf{I}}$ stays within $ \mathcal{Y} $ almost surely, for any $t \geq 0, $ see e.g. \cite{delattre2016hawkes}.
\end{remark}
 }

\subsection{Age dependent Hawkes process with hard refractory period.} \label{hard AHP}

In this section, we are interested in writing a Kalikow decomposition for Age dependent Hawkes processes with hard refractory period, \change{for the purpose of Perfect Simulation of its stationary version}. This process was first introduced in \cite{Chev17}, and no Kalikow decomposition has been established \change{for this process}, even not in a conditional framework. 
In our setting, the stochastic intensity of an Age dependent Hawkes process with hard refractory of length $\delta>0$ can be written as follows. For any $i \in \textbf{I}$ and $x \in \mathcal{X}$,

\begin{equation} \label{eq: hard refractory period}
\phi^{i}(x) = \psi^i \left(\sum_{j \in \textbf{I}} \int_{-\infty}^{0}  h^i_j(- s) dx^j_s \right) \mathds{1}_{a^i(x)> \delta}
\end{equation}
with the \change{age of the process} defined by  $$a^i(x) = - \sup \{t^i_k \in  x \quad \text{such that} \quad t^i_k<0 \},
$$  
\change{that is, the delay since the last point. If there is no such last point, we simply put $a^i(x)= + \infty$. The function $\psi^i$ is called the rate function, whereas the $h^i_j$ are still the interaction functions of the Hawkes process. In this setting we allow for an infinite number of components, that is, $ \bf I$ might be infinite. The set $\mathcal{V}^i$ remains defined by \eqref{Vrondi} but the set $\mathcal{Y}$ is not the same. Indeed, by} definition of the stochastic intensity \eqref{eq: hard refractory period}, we observe that the distance \change{between} any two consecutive jumps \change{has} to be larger than $\delta$. This observation leads us to consider the subspace 
$$\mathcal{Y} = \mathcal{X}^{>\delta}$$ 
which has been introduced in \eqref{Xdelta}. 

\change{We also change the family of neighborhoods. Now we envision a family of nested and increasing neighborhoods. More specifically, from a nested family of subsets of ${\bf I}$ given by $(\omega^i_k)_{k\in \mathds{N}^*}$,  with $\omega^i_1=\{i\}, \omega^i_k\subset \omega^i_{k+1}$ and $\cup_{k\geq 1} \omega^i_k= {\bf I}$, one can design a nested family of neighborhoods ${\bf V}^{i}_{nested}$ by
$${\bf V}^{i}_{nested}=\{ v^i_k= \omega^i_k\times [-k\delta,0) \quad \mbox{for } k\in \mathds{N}^*\}.$$
Note that this family does not include the empty set. This is due to the presence of the hard refractory period. As we will see in Section \ref{sec:4}, the backward steps in the Perfect Simulation algorithm do not end only because of the probability of picking the empty set but also because some neighborhoods might just be empty (in the sense that no points appear in them). This the main difference with similar Perfect simulation algorithms that exist in discrete time, see e.g. \cite{ORB} or  \cite{HL}.

For this family, one can prove the following result, as a  corollary of Proposition \ref{prop:article-2/MethodToWriteKal}.}

\change{\begin{corollary}\label{Cor:age}
Assume that for all  $i, j  \in {\bf{I}},$ ${h}^i_j(.)$ is a non negative, non increasing $L^1$ function and that for every $i,$
$$\sum_{j \in \bf{I}} \| h^i_j\|_{1} < \infty \quad  \mbox{and} \quad \sum_{j \in \bf{I}} h^i_j(0) < \infty .$$
Then the multivariate Age dependent Hawkes process with hard refractory period defined by \eqref{eq: hard refractory period} admits a Kalikow decomposition w.r.t. ${\bf V}^i={\bf V}^{i}_{nested}$ and $\mathcal{Y}=\mathcal{X}^{>\delta},$ under the following conditions.
\begin{enumerate}
\item If for every $i,$  $\psi^i(.)$ is an increasing, non negative continuous function, then the Kalikow decomposition of Proposition \ref{prop:article-2/MethodToWriteKal} Item 1. applies with $\Delta^i_1(x)=\Delta^i_{v^i_1}(x)=\psi^i(0){\bf 1}_{a^i(x) > \delta} $ and for all $k\geq 2$
$$\Delta^i_{k}(x) = \Delta^i_{v^i_k}(x) = \psi^i \left(\sum_{j \in \omega^i_k} \int_{-k\delta}^{0}  h^i_j(- s) dx^j_s \right) {\bf 1}_{a^i(x) > \delta} - \psi^i \left(\sum_{j \in  \omega^i_{k-1}} \int_{-(k-1)\delta }^{0}  h^i_j(- s) dx^j_s \right) {\bf 1}_{a^i(x) > \delta}.$$
\item If in addition $\psi^i(.)$ is $L$-Lipschitz, the Kalikow decomposition of Proposition \ref{prop:article-2/MethodToWriteKal} Item 2. applies for the choices
$\Gamma^i_1=\Gamma^i_{v^i_1}\geq \bar{\Gamma}^i_1:=\psi^i(0)$ and for all $k\geq 2$
$$\Gamma^i_k=\Gamma^i_{v^i_k} \geq \bar{\Gamma}^i_k := L \left[ \sum_{j\in \omega^i_k\setminus \omega^i_{k-1}} \left(h^i_j(0)+\delta^{-1} \| h^i_j\|_{1}\right) + \sum_{j\in\omega^i_{k-1}} h^i_j((k-1)\delta)\right], $$
as long as $\sum_{k=1}^\infty  \Gamma^i_k  <\infty$. 

In particular, since $\sum_{k=1}^\infty  \bar{\Gamma}^i_k  \leq \psi^i(0) + 2 L \left[ \sum_{j\in {\bf I}}h^i_j(0) + \delta^{-1} \sum_{j\in {\bf I}} \| h^i_j\|_{1}\right],$   $\Gamma^i_k=\bar{\Gamma}^i_k$ is a valid choice.
\end{enumerate}
\end{corollary}
}

\begin{proof}


\change{For the first point of the result}, by assumption, 
 $(\Delta^i_k(x))_{k \change{\geq 1}}$ is well-defined, 
non negative and cylindrical on  \change{$v^i_k=\omega^i_k\times[-k\delta,0)$, since the family of neighborhoods is nested}. Let
$$
r^i_n(x) := \sum_{k=1}^n \Delta^i_k(x) = \psi^i \left(\sum_{j \in \omega^i_n} \int_{-n \delta}^{0}  h^i_j(- s) dx^j_s \right) \mathds{1}_{a^i(x) > \delta} 
$$
and let us show that $r^i_n(x) \to \phi^i(x)$ when $n \to \infty$.
Consider the inner-term of the parenthesis, 
$$ 
\sum_{j \in \omega^i_n} \int_{-n \delta}^{0} h^i_j(-s) dx^j_s = \int_{{\bf I} \times \mathds{R}^{-}} h^i_j(-s) \mathds{1}_{(j,s) \in v^i_n} dx^j_s d\kappa_j ,
$$
where we denote $d\kappa$ the counting measure on the discrete set $\textbf{I}$. 

We have that $\left(h^i_j(-s) \mathds{1}_{(j,s) \in v^i_n} \right)_{ n \in \mathds{Z}}$ is a non negative and non decreasing sequence in $n$. In addition, it converges to $h^i_j(-s) \mathds{1}_{(j,s) \in \change{\bf I} \times (-\infty,0)}$ as $n \to \infty$. Moreover, since $\psi^i(.)$ is a continuous and increasing function, the monotone convergence theorem for Lebesgue Stieltjes measures implies that $r^i_n(x) \to \phi^{i} (x)$ as $n \to \infty$. As a consequence,  \change{Proposition \ref{prop:article-2/MethodToWriteKal}.1 applies.}

\change{For the second part, $\psi^i$ is $L$-Lipschitz. Hence we have
\begin{equation}\label{eq1}
\Delta_i^k(x) \leq L \times \left[\sum_{j\in \omega^i_k\setminus \omega^i_{k-1}} \int_{-k \delta}^{0} h^i_j (-s) dx^j_s + \sum_{j\in \omega^i_{k-1}} \int_{-k\delta}^{-(k-1)\delta} h^i_j (-s) dx^j_s\right].
\end{equation}

So  let us fix $0 \leq k < l , j \in \textbf{I}$, $x \in \mathcal{X}^{>\delta}$ and $t \in \mathds{R}$ and let us concentrate first  on upper bounding $\int_{t - l\delta}^{t- k \delta}h^i_j(t-s)dx^j_s$ by adapting}
Lemma 2.4 of \cite{RDL}. For any $\epsilon > 0$, we have

\begin{align*} 
\int_{[t- l \delta, t- k \delta-\epsilon]}h^i_j(t-s)dx^j_s &= \sum_{k \leq m < l} \int_{[t- (m+1)\delta - \epsilon, t- m \delta - \epsilon]} h^i_j(t-s)dx^j_s \\
&\leq \sum_{k \leq m < l} {h}^i_j(m \delta + \epsilon),
\end{align*}
\change{because $h^i_j$ is non increasing and because} there is at most one jump in the interval of length $\delta$.
Therefore, 
\begin{eqnarray*}
\int_{[t- l \delta, t- k \delta)}h^i_j(t-s)dx^j_s &= & \lim \limits_{\epsilon \downarrow 0} \int_{[t- l \delta, t- k \delta- \epsilon]}h^i_j(t-s)dx^j_s \\
                                                             &\leq& \lim \limits_{\epsilon \downarrow 0} \sum_{k \leq m < l} {h}^i_j(m \delta + \epsilon) 
                                                            \leq \sum_{k \leq m < l} {h}^i_j(m \delta),
\end{eqnarray*}
by  monotone convergence  and the fact that ${h}^i_j(.)$ is a decreasing function. 

\change{Going back to \eqref{eq1}, we have therefore that
$$ \int_{-k\delta}^{-(k-1)\delta} h^i_j (-s) dx^j_s \leq {h}^i_j((k-1)\delta))$$
and
$$  \int_{-k \delta}^{0} h^i_j (-s) dx^j_s \leq \sum_{0 \leq m \leq  k-1} {h}^i_j(m \delta) \leq h^i_j(0) + \delta^{-1} \int_0^{k-1} h^i_j(t)dt \leq  h^i_j(0) + \delta^{-1} \|h^i_j\|_1.$$
This shows that if $\Gamma^i_k\geq \bar{\Gamma}^i_k$, we are indeed upper bounding $\sup_{x \in \mathcal{X}^{>\delta}} \Delta^i_k(x)$. The last upper-bound on $\sum_k \bar{\Gamma}^i_k$ is done in a similar way.}
\end{proof}

\begin{remark}
\change{In particular this proves that $ \Delta^i_k(x) \to 0$ when $k \to \infty$ for fixed $x \in \mathcal{X}^{>\delta},$ and that this convergence is uniform on $ \mathcal{X}^{>\delta},$  if $\psi^i$ is Lipschitz.}
\end{remark}

\begin{remark}
\change{It is straightforward to see that}
\begin{align*}
\inf_{z \overset{v^i_k}{=} x } \psi^i \left(\sum_{j \in {\bf I}} \int_{-\infty}^{0}  h^i_j(- s) dz^j_s \right) \mathds{1}_{a^i(x) > \delta} = \psi^i \left(\sum_{j \in \omega^i_k} \int_{-k \delta}^{0}  h^i_j(- s) dx^j_s \right) \mathds{1}_{a^i(x) > \delta},
\end{align*}
\change{because of the monotonicity property of $\psi^i$ and since $h^i_j \geq 0.$ Indeed the minimizing configuration is obtained by having no points outside of $v^i_k$.} Hence, for $k \geq 2$, we observe that
$$
\Delta^i_k(x) =  \inf_{z \overset{v^i_k}{=} x } \left[\psi^i \left(\sum_{j \in {\bf I}} \int_{-\infty}^{0}  h^i_j(- s) dz^j_s \right) \mathds{1}_{a^i(x) > \delta} \right]
- \inf_{z \overset{v^i_{k-1}}{=} x }\left[\psi^i \left(\sum_{j \in{\bf I}} \int_{-\infty}^{0}  h^i_j(- s) dz^j_s \right) \mathds{1}_{a^i(x) > \delta}\right] .
$$
The above prescription corresponds to the classical method of obtaining a Kalikow decomposition in discrete time, as discussed in \cite{HL, GL2013} (see also Corollary \ref{thinn_kali}).
\end{remark}

\begin{remark}
\change{ Notice that it is not possible to let  $\delta$ tend to 0 and to recover Corollary \ref{Cor:age} for linear or even general non-linear Hawkes processes having Lipschitz continuous rate function, since typically the bounds in Corollary \ref{Cor:age} are exploding. }
\end{remark}

\subsection{\change{Analytic rate function in a Hawkes process \label{sec_nonlin}}}
\change{In the previous examples, we have shown how one can easily derive the Kalikow decomposition if the intensity is already in the shape of a sum, as it is the case for the linear Hawkes process, with an atomic family of neighborhoods. In this case the complexity of all neighborhoods is very small. In the case of the Age dependent Hawkes process, thanks to the refractory period and the monotonicity and continuity properties of the underlying functions, we have been able to derive a Kalikow decomposition as well, however with respect to a more complex family of neighborhoods which is nested. Moreover we have shown that under Lipschitz properties, the intensities are bounded. In both cases the processes were non exploding, that is, they only have a finite number of jumps within finite time intervals, almost surely. 

Now we want to reach much more erratic processes, that can even explode in finite time. To do so, we cannot keep the refractory period and we consider} nonlinear Hawkes processes $Z = (Z^i)_{i \in \textbf{I}}$ with intensity given by  
\begin{equation}\label{nonlin}
\phi^{i}(x) = \psi^{i} \left(\sum_{j \in \textbf{I}} \int_{-\infty}^{0} h^i_j(-s) dx^j_s \right)
\end{equation}
for any $x \in \mathcal{X} ,$ where $ \psi^i : \mathds{R}_+  \to \mathds{R}_+$ are measurable functions and where $ h_j^i : \mathds{R} \to \mathds{R}_+$ belong to $L^1.$ 
\change{

In \cite{BM96} it has been proved that if $\psi^i$ is Lipschitz, then -- under suitable additional assumptions on $ \| h^i_j\|_1 $ -- a stationary version exists. However, in the case where $\psi^i$ is only locally Lipschitz, the existence of a non-exploding or even a stationary solution is not guaranteed. For example, when choosing $ \psi^i ( x) = e^x, $  then the process may  explode in finite time with strictly positive probability,  see \cite{Cars}.}

\change{We are going to use analytical properties of $\psi^i$ to perform a Taylor expansion of $\psi^i.$ To do so, we are considering a family of neighborhoods which is a sort of  iterated tensor products of ${\bf V}_{atom}$. For a fixed $\epsilon$, recall that an atomic neighborhood (see \eqref{atomic}) is such that
$w_{j,n}=\{ j \} \times [-n\epsilon, -(n-1)\epsilon)$, for $j\in {\bf I}$ and $n\in\mathds{N}^*$.
A neighborhood of order $k$ is then a union of $k$ atomic neighborhoods. For $\alpha_1=(j_1,n_1)$,..., $\alpha_k=(j_k,n_k),$ we put
$$v_{\alpha_{1:k}}=\bigcup_{k=1}^nw_{j_k,n_k}.$$
The family of all possible such unions is denoted ${\bf V}_{\otimes k}$, with the convention that ${\bf V}_{\otimes 0}=\{\emptyset\}$. 

The Taylor family of neighborhoods is then defined by
$${\bf V}_{Taylor}= \bigcup_{k=0}^\infty {\bf V}_{\otimes k}.$$

Note that if for instance, some of the $\alpha_i$'s are equal, then the neighborhood $v_{\alpha_{1:k}}$ collapses  on a smaller union of less than $k$ intervals and it is also possible that two neighborhoods with different indices are in fact equal. We do not  simplify the family and remove redundancies. It is important that neighborhoods with different indices are considered different to define properly the probability $\lambda^i$.

We are limited by the radius of convergence of the function $\psi^i$, say $K$. This is why we are working within the space 
$$\mathcal{Y}=\mathcal{X}^{<K} = \{ x \in \mathcal{X} : \sup_{i \in {\bf I}} \sum_{j \in {\bf I}} \int_{- \infty}^0 h^i_j (-s) d x_s^j < K \}.$$}

\change{\begin{corollary}
Let us assume that for any $i,j \in {\bf I}$, the function $h^i_j(.)$ is non negative and that for every $i$,
\begin{equation*}
\sup_t\sum_{j \in {\bf I}} h^i_j(t) <\infty.
\end{equation*}
Let us also assume that for every $i \in {\bf{I}}$, $\psi^{i}(.)$ is an analytic function on $\mathds{R}$ with radius of convergence $K > 0 $ around  $0$, such that its derivative of order $k$, $[\psi^i]^{(k)}(0),$ is non negative for all $k\geq 0.$
Then the multivariate non linear Hawkes process defined by \eqref{nonlin} admits the  Kalikow decomposition of Proposition \ref{prop:article-2/MethodToWriteKal}.1 w.r.t to ${\bf V}^i= {\bf V}_{Taylor}$ for all $i$ and with respect to the subspace $\mathcal{Y}=\mathcal{X}^{<K}$, with
\begin{itemize}
\item if $v=\emptyset$, $\Delta^i_v(x)= \psi^i(0).$
\item if $v=v_{\alpha_{1:k}}$, for $\alpha_1=(j_1,n_1)$,..., $\alpha_k=(j_k,n_k)$, then
$$\Delta^i_v(x) = \frac{[\psi^i]^{(k)}(0)}{k !} a_{\alpha_1}(x) \cdot \ldots \cdot a_{\alpha_k}(x),$$
where for $\alpha=(j,n)$, 
$$a_{\alpha}(x) := \int_{-(n+1)\epsilon}^{-n\epsilon} h^i_j(-s) dx^j_s .$$
\end{itemize}
\end{corollary}
\begin{proof}
The condition ensures that the intensities are at least locally defined if, starting from an empty past, one first point occurs. The proof of the convergence is then straightforward and consists in using the analytic properties of $\psi^i$.
\end{proof}

Note that the linear Hawkes process is a particular case of this Corollary with $\psi^i(u)=\mu^i+u$ and $K=\infty$.
But the above strategy can also be applied to  the exponential Hawkes process, $\psi^i(u)=\exp(u)$, $K=\infty$ and $[\psi^i]^{(k)}(0)=1$ for all $k$, which is only locally Lipschitz. One can also use it for other functions, for instance $\psi^i(u)=\mathrm{ch}(u)= (e^u+e^{-u})/2$ with $K=\infty$, for which even derivatives $[\psi^i]^{(2k)}(0)=1$ and for which the other derivatives are null.}

\subsection{Galves-L\"ocherbach processes with saturation thresholds}
\change{We close this section with an example which is not directly related to Hawkes processes and for which, using the same arguments as above, we may establish a Kalikow decomposition as well. This  is the model with saturation threshold which has already been considered in \cite{HL}. }
\change{More precisely, we put 
\begin{equation}\label{eq:GL}
 \phi^i ( x)  = \psi^i \left( \sum_{j \in \bf{ I}}  \left( \beta_j^i  \; Z^j ( ( - a^i (x), 0)) \right) \wedge K_j^i \right)  , 
\end{equation} 
for $ \psi^i : \mathds{R} \to \mathds{R}_+ $ a Lipschitz continuous non-decreasing rate function, $ \beta_j^i  \geq 0 $ the weight of $ j $ on $ i ,$ where we choose $\beta^i_i = 0,$  and $ K_j^i \geq 0 $ the saturation threshold. In this framework, analogously to Corollary \ref{Cor:age}, we can prove the following }
\begin{corollary}\label{GL-model}

Consider the Galves-L\"ocherbach process having generic intensity $ \phi^i $ given by \eqref{eq:GL} and suppose that 
$$ \sup_{ i \in \bf{ I}}  \sum_{j \in \bf{ I}} K_{j}^i   < \infty .$$ 
Then $\phi^i $ admits the Kalikow decomposition given by Proposition \ref{prop:article-2/MethodToWriteKal} Item 1., with respect to  the collection of neighborhood families $({\bf V}^i)_{i \in {\bf I}}$ where for all $i$, ${\bf V}^i={\bf V}^{i}_{nested}\cup \{v^i_0=\emptyset\},$ and with respect to the subspace $\mathcal{Y}= \mathcal{X} ,$ with $\Delta^i_0 ( x) = \psi^i ( 0 )$ and for any $ k \geq 1, $ 
$$\Delta^i_{k}(x) = \Delta^i_{v^i_k}(x) = \psi^i \left(\sum_{j \in \omega^i_k} [\beta_j^i \int_{-(k\delta) \wedge a^i (x) }^{0}  dx^j_s] \wedge K^i_j  \right) - \psi^i \left(\sum_{j \in  \omega^i_{k-1}} [\beta_j^i \int_{-((k-1)\delta)\wedge a^i ( x) }^{0} dx^j_s] \wedge K^i_j  \right) ,$$
where $ w^i_0 =\emptyset.$
\end{corollary}

\section{Simulation algorithms}\label{sec:4}
\change{In this section we present two types of algorithms corresponding to the two items of Proposition \ref{prop:article-2/MethodToWriteKal}. 
First, a simulation algorithm that simulates the time homogeneous counting process $ Z$ for finite networks, starting from the empty past before time $0,$ up to some finite time, under  the conditions of Item 1. of Proposition \ref{prop:article-2/MethodToWriteKal}. And secondly, a Perfect Simulation algorithm that simulates the process $ Z$ in its stationary regime, within a finite space-time window, under the conditions of Item 2. of Proposition \ref{prop:article-2/MethodToWriteKal}. }

\subsection{Simulating forward in time, starting from a fixed past}\label{Sec:forward}
\change{Suppose that we are in the situation of Item 1. of Proposition \ref{prop:article-2/MethodToWriteKal}, with $ \bf{I} $ finite and that we wish to simulate the process until a certain $t_{max}>0$ starting from the empty past before time $0.$ We might never reach $t_{max}$ because the process might typically explode, so we introduce
$$ \tau := t_{max} \wedge \inf \{ t \geq 0 : Z_t \notin \mathcal{Y} \}.$$
For simplicity we also suppose that in absence of further jumps, the intensities are non-increasing. More precisely, introducing for all $i \in \bf{I}, $ for all $ t \geq 0, $ for any $x \in \mathcal{X}, $ the configuration where all jumps of $ x$ are shifted to the left by $t $  and no further jumps are added within $ (- t, 0 ), $ which is denoted by   $ ( x^{ \leftarrow t },  {\bf 0} ), $ we suppose 

\begin{assumption}
For all $ i \in \bf{I}, $  for all $v \in \bf{V}^{i}  $ and for all $ t \geq 0, $ 
$$  \phi^i_v ( (x^{ \leftarrow t},  {\bf 0}) ) \le \phi^i_ v (x).$$   
\end{assumption}
This is typically the case for the linear or the exponential Hawkes processes with decreasing interaction functions $h^i_j$.

Let us now describe how we simulate forward in time, starting from the empty past before time $0,$ up to some maximal number of simulation steps $N_{max} > 0 .$} \\

\change{
{\bf Step 0.} We initialize the set of points with $ X = {\bf 0} \in \mathcal{X} ,$ where $ {\bf 0} \in \mathcal{X}$  designs the configuration having no points,  and $ T = N= 0 .$  \\

\noindent While $ T \le \tau $ and $ N \le N_{max}$  do the following. \\

{\bf Step 1.} For each $ i \in \bf{I}, $ compute $ \tilde{\Gamma}^i \geq  \sup_{ v \in \textbf{V}^{i}} \phi_v^i ( X^{\leftarrow T}). $ \\

{\bf Step 2.} Simulate the next jump $ T \leftarrow T + Exp ( \sum_{i \in \bf{I}} \tilde{\Gamma}^i ) .$\\
If $T>\tau$ stop, else choose the associated index $ i$ with probability 
$$ \mathds{P} ( I = i ) = \frac{\tilde{\Gamma}^i }{\sum_{j \in \bf I} \tilde{\Gamma}^j }. $$ 

{\bf Step 3.} Choose the associated interaction neighborhood $ V_T= v $ with probability $ \lambda^i ( v).$ \\

{\bf Step 4.} Accept $ T  $ as a jump of index $ i $ with probability $ \frac{ \phi^i_v (X^{\leftarrow T })}{\tilde{\Gamma}^i }$ and add $ ( i, T) $ to $X$ in this case. Put $ N \leftarrow N +1.$ If the point is rejected,  do nothing. Go back to Step 1. 
 }

\change{\begin{remark}
The bound $\tilde{\Gamma}^i ,$ computed in {\bf Step 1}, is random since it depends on the configuration $X$ that has been simulated so far.
\end{remark}}

\change{\begin{corollary} [of Proposition \ref{thinn_kali}]
The previous algorithm simulates a time homogeneous counting process $Z$ of generic intensity given by  Item 1. of Proposition \ref{prop:article-2/MethodToWriteKal}.
\end{corollary}}

\begin{proof}
\change{ Note that the random variable $\tilde{\Gamma}^i$ is a bound on the generic intensity $\phi^i$ in absence of appearance of a new point in the future after time $T$. Hence the ``new" $T$ (see also Section \ref{thinning}) computed at {\bf Step 2} can be seen as  the abscissa of a point of an hidden bivariate Poisson process in a band of height  $\sum_{i \in \bf{I}} \tilde{\Gamma}^i$. The association of a particular index $i$ is quite usual (see for instance Ogata's algorithm \cite{Ogata}) and it means that if the chosen index is  $i,$  then the point $T$ can also be seen as the next point of the bivariate Poisson process $\pi^i$ that is discussed in Section \ref{thinning} in the band of height $\tilde{\Gamma}^i$. Since this is an upper-bound on $\phi^i,$ the thinning procedure will not consider points of $\pi^i$ outside this band, at least until a new point is accepted. The association of a neighborhood in {\bf Step 3} is therefore consistent with the tridimensional Poisson process of Proposition \ref{thinn_kali}, and the points that are accepted at {\bf Step 4} are the ones accepted  in Proposition \ref{thinn_kali}. Therefore this algorithm indeed simulates points with the desired intensities as long as the the process stays in $\mathcal{Y}$ and the Kalikow decomposition holds.}
\end{proof}

\change{\begin{remark}
Note that this algorithm will indeed save computational time if both the computation of $\tilde{\Gamma}^i$ and  the verification that one stays in $\mathcal{Y}$  is easy. In the particular example of Linear Hawkes process (Corollary \ref{linHaw_kali}),  with decreasing $h^i_j$'s, and bounded support say on $[0,\epsilon]$, one can take $\tilde{\Gamma}^i=h^i_j(0) Z^i([T-\epsilon,T))$ and we know that the process will stay in $\mathcal{Y}$ with probability $1$. Further algorithmic work, beyond the scope of the present article, might include a real handcraft on the design of  $\tilde{\Gamma}^i$ to save computational time.
\end{remark}}

\subsection{Perfect simulation of the process in its stationary regime}\label{sec:41}
\change{Throughout this section we suppose that we are in the situation of Item 2. of Proposition \ref{prop:article-2/MethodToWriteKal}, that is the generic intensities $\phi^i$'s are upperbounded by deterministic constants $\Gamma^i$  and that the subspace $\mathcal{Y} $ is invariant under the dynamics, that is, $ Z_t \in \mathcal{Y} $ implies  $Z_{t+s} \in \mathcal{Y} $ for all $ s \geq 0.$ This is e.g. the case for the Age dependent Hawkes process with hard refractory period considered in Section \ref{hard AHP} above. In particular, in this case, it is possible to restrict the dynamic to $\mathcal{Y} ,$ and in what follows we will show how to simulate from the unique stationary version of the process within this restricted state space $ \mathcal{Y}.$} 

In what follows, we propose to simulate, in its stationary regime, the process $Z^i$ for a fixed $i \in \bf{I},$ on an interval $[0, t_{max}],$ for some fixed $ t_{max} > 0 .$ Our algorithm is a modification of the method described in \cite{PMR}, which works in the case where all the $\Gamma^i$'s are equal. The procedure consists of backward and forward steps. In the backward steps, thanks to the Kalikow decomposition, we create a set of ancestors, which is a list of all the points that might influence the point under consideration. On the other hand, in the forward steps, where we go forward in time, by using the thinning method \cite{Ogata} in its refined version stated in Theorem \ref{thinn_kali}, we give the decision (acceptance or rejection) to each visited point based on its neighborhood until the state of all considered points is decided. 

\change{
The idea of relying on such a two-step procedure is not new and has already been proposed in the literature, even in a continuous time setting, see e.g. \cite{pablo}, \cite{GLO} and \cite{GGLO}  where such an approach is used to  simulate from infinite range Gibbs measures and/or from the steady state of  interacting particle systems, relying on a decomposition of the spin flip rates as a convex combination of local flip rates. The main difference of our present approach with respect to these results is twofold. Firstly, in these articles, only spatial interactions are considered -- while we have to decompose both with respect to the spatial interactions and the history in the present article. And secondly and more importantly, in all these articles the authors manage to go back to transition probabilities -- and then establish a Kalikow decomposition for these probabilities. So somehow this means that we are back in the framework of discrete time processes as in \cite{Kal90} where these ideas have been introduced. Such a discrete time approach is still used in \cite{HL}. The idea of decomposing directly the intensities, not to go back to probabilities, and to decompose  both with respect to time, that is, history, and space, that is, the interactions, is -- at least to our knowledge -- completely new. }

\subsubsection{Backward procedure}

\change{Recall that according to Item 2. of Proposition \ref{prop:article-2/MethodToWriteKal}, all functions $ \phi^i$ and $ \phi^i_v $ are bounded by a fixed deterministic constant $ \Gamma^i.$

For the sake of understandability from a theoretical point of view, we will assume that the $N^i$'s are independent homogeneous Poisson processes on the real line with intensity $\Gamma^i$ and that these points are fixed before one starts the algorithm. Indeed following what has been said in Proposition \ref{thinn_kali}, we will recover the points of $Z^i$ as a thinning of a tridimensional Poisson process. The fact that the intensities are bounded allows us to thin only in a band of height $\Gamma^i$, and that is why we are going to thin $N^i$ into $Z^i$. 

Of course in practice, the infinite knowledge of the $N^i$'s is not possible and the {\bf Step 1'} and {\bf Step 2'} are there to explicit concretely how this is done in practice. 

\bigskip

\textbf{Step 0.} Fix $i$, set the initial time to be $T_0=0$. 

\bigskip

\textbf{Step 1.} Take $T$ the first point of $N^i$ after 0. This is the first possible jump  of $Z^i$ after $T_0$. 

\textbf{Step 1'.} Simulate this $T$ as 
$$ 
T \leftarrow T_0 + Exp(\Gamma^i).
$$
In particular this implies that the underlying $N^i$ is empty on $[T_0,T)$.\\
If $T>t_{max}$, stop.

\bigskip

\textbf{Step 2.} Independently of anything else, pick a random neighborhood $V^i_T$ of $(i,T)$ according to the distribution $(\lambda^i  (v))_{v \in \textbf{V}^{i}}$ given in  \eqref{eq:choicelambda}, that is 
$$
\mathds{P}(V^i_T = v) = \lambda_{i}(v).
$$
On $ \{V^i_T = v\}, $ consider the shift of this neighborhood $v$  by $T$ defined by  
$$ v^{ \rightarrow T} := \{ (j , u + T ) : (j, u ) \in v\} $$ 
and, for any $ j \in \bf{I}, $ the projection to the second coordinate of $v^{ \rightarrow T} $ by
$$
p_{j}(v^{ \rightarrow T} ) := \left\{t \in \mathds{R} :  (j,t) \in v^{ \rightarrow T}  \right\} .
$$
Notice that if for some $j, $  $(j,t) \notin v^{ \rightarrow T}$ for all $t,$ then $p_{j}(v^{ \rightarrow T}) = \emptyset.$
Finally, we look for the set of points of $N$ that might directly influence the decision of acceptation/rejection of $T$:
$$
\mathcal{C}_{1}^{(i,T)} = \bigcup\limits_{j \in \textbf{I}} \left\{(j,t) :  t \in  p_j(v^{\rightarrow T})  \mbox{ is a jump of } N^j   \right\}.
$$

\textbf{Step2'.} Simulate independent Poisson processes in  $v^{\rightarrow T}$, that is for each $j \in \textbf{I}$, we simulate the points of  $N^j$  independently of anything else as a Poisson process of intensity $\Gamma^j$ on $p_{j}(v^{\rightarrow T})$ to identify as before $
\mathcal{C}_{1}^{(i,T)}$.

It is possible that at this point in the algorithm (especially in the iteration below), $v^{\rightarrow T}$ intersects with neighborhoods that have already been picked. In this case, simulate only on the portion of the neighborhood that has never been visited before. Note in particular that we take for already visited 
region the set $\{i\}\times[T_0,T),$ and that within this region there are no points by definition of $ T.$

\bigskip

\textbf{Step 3.} Recursively, we define the $n^{th}$ set of ancestors of $(i,T)$ by 
$$
\mathcal{C}_{n}^{(i,T)} = \bigcup_{(j,s) \in \mathcal{C}_{n-1}^{(i,T) }} \mathcal{C}_{1}^{(j,s)} \setminus \left(\mathcal{C}_1^{(i,T)} \cup \ldots \cup \mathcal{C}_{n-1}^{(i,T)} \right) ,
$$
by performing {\bf Step 2} or {\bf Step 2'}  for each $(j, s) \in \mathcal{C}_{n- 1}^{(i,T)}.$ 

Note that $
\mathcal{C}_{n}^{(i,T)}$  exactly corresponds to the points that are really simulated as new in the iterated {\bf Step 2'}.}

\bigskip

We denote 
$$
N^{(i,T)} = \inf \{n: \mathcal{C}_n^{(i,T)} = \emptyset \}, 
$$
where $\inf \emptyset := + \infty.$ \change{The backward scheme stops when $N^{(i, T ) } < \infty, $ and we give below sufficient conditions guaranteeing this fact (see Proposition \ref{prop:article-2/Stopping condition} below). In this case we say that the total clan of ancestors  of $(i,T)$ is given by 
$$
\mathcal{C}^{(i,T)} = \bigcup_{k =1}^{N^{(i,T)}} \mathcal{C}_{k}^{(i,T)}
$$}
In what follows we also consider the associated interaction support given by 
$$ 
\mathcal{V}^{(i, T) } = \bigcup_{(j, t ) \in \mathcal{C}^{(i,T)} } (V^i_t)^{\rightarrow t } $$ 
and we put 
\begin{equation}\label{eq:tit}
 T^{(i, T )} :=T-  \inf \{ s : \exists j \in {\bf{I}} : (j, s) \in \mathcal{V}^{(i,T)}  \}
\end{equation} 
which is the total time the backwards steps need to look back in the past.

\begin{remark}
Let us emphasize that $\lambda_i(\emptyset)$ does not need to be strictly positive in order to guarantee that $N^{(i, T ) } < \infty. $  If $\lambda_i(\emptyset)=0$, at every step of the Backward scheme, we always need to simulate a Poisson process in a non empty neighborhood. However, if there is no point simulated in such a neighborhood, then we do not add any points to the clan of ancestors, and if this happens sufficiently often, then the Backward step ends as well. This is one of the main advantages of our approach and a major difference with respect to \cite{HL,GL2013}. \change{We give in Proposition \ref{prop:article-2/Stopping condition} below sufficient conditions implying that the algorithm stops after a finite number of steps almost surely. }
\end{remark}

\change{\begin{remark}
The steps {\bf Step 1'} and {\bf Step 2'} are consistent with recreating the processes $N^i$'s in the neighborhoods of interest, instead of taking them for granted beforehand. In particular this is the reason why the algorithm is very careful when dealing with overlapping neighborhoods to not simulate twice the same process on the same portion of the space.
\end{remark}}

\subsubsection{Forward procedure}
\change{Supposing that $N^{(i, T ) } < \infty, $} we now use a Forward procedure to \change{accept or reject recursively} each point in $\mathcal{C}^{(i,T)}$\change{, until the status of $T$ is decided.} 

We start with the point $(j,s) \in \mathcal{C}^{(i,T)}$ which is the smallest in time, so that its associated neighborhood is either empty $(v = \emptyset)$ or non empty but without any point of the Poisson process in it.

\textbf{Step 1.} \change{Accept  $(j,s)$} with probability $\dfrac{\phi_{V^j_s}^{j}(X^{\leftarrow s}_s)}{\Gamma^j}$ where $V^j_s$ is the neighborhood  of $(j,s)$ . 

\bigskip

\textbf{Step 2.} Move to the next point of $\mathcal{C}^{(i,T)}$ in increasing time order. Repeat Steps 1 and 2 until \change{the status of $T$} is determined.

\bigskip

\noindent \textbf{Update step.}
\change{To simulate on $[0,t_{max}]$, go back to {\bf Step 0} of  the Backward procedure and replace the starting time of the initial step $T_0$ by $T$. Repeat the Backward and Forward procedures until  $T>t_{max}.$} 

\begin{remark}
If one wants to adapt the algorithm to the simulation of a generic  finite subset $F$ of ${\bf I}\times \mathds{R},$ it is sufficient to shift everything to the smallest time in $F$ (which will be the initial  $T_0$) and to repeat the process on all the $i$'s in $F$. Again in {\bf Step 2'}, we need to be careful to simulate only on new parts of ${\bf I}\times \mathds{R}$ and not parts that have been already discovered and where one of the  $N^i$ has already been simulated.
\end{remark}

\subsubsection{\change{Why  the Backward procedure ends after a finite number of steps}}\label{sec:end}
Our procedure only works if $N^{(i, T ) } < \infty $ almost surely. In what follows we give sufficient conditions implying this. To do so,  
\change{we compare our process to a spatio-temporal branching process. Let us define this more mathematically.}

\bigskip

\change{\underline{\it The initial directed graph of ancestors $\mathcal{T}$}}

 We start with the ancestor which is the point $ (i, T)$ constituting generation $0.$ All points belonging to $ \mathcal{C}_n^{(i, T ) } $ are called points of the $n-$th generation. For each point $(j,t )$ belonging to generation $n,$ all elements of 
$$\mathcal{C}_1^{(j, t ) }  = \bigcup\limits_{k \in \textbf{I}} \left\{(k,s) :  s \in  p_k((V^j_t)^{\rightarrow t})  \mbox{ is a jump of } N^k   \right\}$$ 
define the children of the point $(j,t).$ The choices of the children of $ (j, t) $ and $ (j', t') $ for any two distinct elements of generation $n$ are not necessarily independent, since the associated neighborhoods might overlap, that is, we may have $ V^j_t \cap V^{j'}_{t'}  \neq \emptyset . $ 

\change{Note that $\mathcal{T}$ is not a tree, precisely because of this potential overlapping: two different parents might have the same child.}

\bigskip

\change{\underline{\it The dominating branching process $\tilde{\mathcal{T}}$}}

To obtain sufficient conditions implying that $N^{(i, T ) } < \infty ,$ we therefore  construct a dominating spatio-temporal branching process starting from the same ancestor $(i,T),$ where the choices of children are independent. To go from generation $n$ to the next generation $n+1, $ to any point $(j,t ) $ belonging to generation $n$ we associate the same children $ \mathcal{C}_1^{(j, t ) } $ as before whenever the chosen neighborhood $V^j_t  $ does not overlap with parts of $ \bf{I} \times \mathds{R} $ where we have already simulated in previous steps. However, if there are parts of the neighborhood $V^j_t    $ that intersect with already visited parts, we  simulate -- independently of anything else --  on the whole neighborhood. By doing so, we make the number of children larger, but for each point, these choices are independent of anything else. If this dominating branching process goes extinct in finite time, then so does the original process $ \mathcal{C}_n^{(i, T)}, $ and the backward part of the algorithm terminates. 

\bigskip

To formulate a sufficient criterion that implies the almost sure extinction of the dominating branching process $\tilde{\mathcal{T}}$,  let us denote the product measure $P$ on $\textbf{I} \times \mathds{R}$, defined on the 
Borel subsets of $\textbf{I} \times \mathds{R},$ as follows.

$$
P(J \times A) := \sum_{j \in J} \Gamma^j \, \mu(A)  
$$
for any $J \subset \textbf{I}$, where $A$ is a Borel subset of $\mathds{R}$ and where $\mu$ is the Lebesgue measure. The following proposition is already proved in \cite{PMR} in a particular case where all the $\Gamma^j$'s are equal.

\begin{proposition} \label{prop:article-2/Stopping condition}
If \change{
\begin{equation} \label{nbenfant}
\sup_{i \in \bf{I}} \sum_{v \in \bf{V}^{i}} P(v) \lambda^i(v) =: \gamma < 1  ,
\end{equation}}
then the Backward steps in the Perfect Simulation  algorithm terminate almost surely in finite time.
\end{proposition}
\change{
\begin{proof}
For any neighborhood $v$, we have
$$
\sum_{j \in v} \mathds{E} \left( \text{card} \left\{  t:  t \in  p_j(v)  \mbox{ is a jump of } N^j   \right\}\right)  =\sum_{j \in v}  \Gamma^j \mu ( p_j(v) ) =    P(v) .
$$
This implies that $\sum_{v \in \textbf{V}^{i}} P(v) \lambda^i(v) $ is the mean number of children issued from one point of type $i$. Then the condition \eqref{nbenfant} implies that the mean number of children is less than one in each step, which is the a classical sub-criticality condition for branching processes, see \cite{AN}.
The result then follows from the fact that according to the above discussion, we can dominate $L_n^{(i, T )} := \text{card} ( \mathcal{C}_n^{(i, T) } ) $ by a classical Galton Watson branching process having offspring mean $\gamma < 1$ which goes extinct in finite time almost surely. 
\end{proof}  }

\subsubsection{Why we sample from the stationary distribution}

\begin{thm}\label{theo:stationary}
\change{Suppose we are in the situation of Item 2. of Proposition \ref{prop:article-2/MethodToWriteKal} and the subspace $\mathcal{Y} $ is invariant under the dynamics, that is, the configuration associated to $Z$, $X$, satisfies that $X_t^{\leftarrow t} \in \mathcal{Y}$ 
 implies $X_{t+s}^{\leftarrow t+s} \in \mathcal{Y}$, for all positive $s$. If \eqref{nbenfant} holds,  then the process possesses a unique stationary distribution in restriction to $\mathcal{Y},$ and the accepted points of the forward procedure yield a perfect sample of this stationary distribution within the space-time window $ \{ i \} \times [0, t_{max}].$ 
}
\end{thm}

\begin{proof}\change{
In this proof we adapt the ideas of the proof of Theorem 2 in \cite{GGLO} to the present framework.
For any $F$ a finite subset of ${\bf I}\times\mathds{R}$ the Backward and Forward procedures produce a sample of a point process within the space-time window $ F,$ and we write $\mu_F$ for the law of the output. By construction, the family of probability laws $ \{ \mu_F, \,  F \subset  \bf{I} \times \mathds{R}  \mbox{ finite} \} $ is a consistent family of finite dimensional distributions. Hence there exists a unique probability measure $ \mu $ on $ ( \mathcal{X}_\infty, {\cal B} (\mathcal{X}_\infty))$ such that $ \mu_F $ is the projection onto $F$ of $\mu, $ for any fixed finite set $F \subset  \bf{I} \times \mathds{R}.$  

We show that $\mu$ is the unique stationary distribution of the process $Z $ within $\mathcal{Y}.$ 
In order to do so, we use a slight modification of our algorithm  in order to 
construct $Z$ starting from some fixed past, say $ x \in \mathcal{Y}, $ before time $0.$ The modification is defined
as follows. 

We fix $ t,  t_{max} > 0 $ and put $ F = \{ i \} \times [t, t + t_{max}].$  Recall that the original Backward procedure relies on the a priori realization of all the Poisson processes $ N^i $ on $ (- \infty , t + t_{max} ].$ In our modified procedure we replace, for any $ i \in \bf{I}, $ the points of $ N^i $ within $ ( - \infty , 0) $ by those of $ x^i, $ where $ x = (x^i )_{i \in \bf{I}}  $ is our initial condition. 

{\bf Step 0.} Put $T_0 = t.$ 

{\bf Step 1.} Perform {\bf Steps 1-3} of the Backward procedure replacing the Poisson processes $ N^i $ in restriction to $ ( - \infty , 0)$ by the corresponding points of $ x,$
and stop this procedure at time  
$$   \tilde N^{(i, T )}  = \inf \{ n : \mathcal{C}_n^{(i, T)} \subset {\bf{I}} \times (- \infty, 0) \} \wedge N^{(i, T )} .$$


Indeed, on the set $\{ \tilde N^{(i, T )}  <  N^{(i, T )}\}, $ at this time, only points having negative times have to be considered, and all these points are determined by the initial condition $x.$  In this modified version, when we stop the algorithm, the output set $\mathcal{C}_{\tilde N^{(i, T )} }^{(i, T )}$ might not be empty. This set is exactly the set of points before time $0$ that have an influence on the acceptance or rejection of the point $ ( i, T) .$ 

Notice that the time 
$$ \inf \{ n :  \exists (j, t ) \in \mathcal{C}_n^{(i, T ) } : (V^j_t)^{\rightarrow t }  \cap {\bf{I}} \times (- \infty, 0) \neq \emptyset \} ,$$
if it is finite, 
is the first time where the modified algorithm starts to be different from the original one. In particular, if the backward steps stop before reaching the negative sites, that is, if we are on the event $  T^{(i, T )} \leq T $ (recall \eqref{eq:tit}), then $\mathcal{C}_{\tilde N^{(i, T )} }^{(i, T )} = \emptyset, $ $ \tilde N^{(i, T )}  =   N^{(i, T )} ,$ and the two procedures produce the same sets of points. 

The forward procedure is performed as before, replacing the unknown points before time $0$ by the fixed past configuration $x.$ 

Then the law of the set $\{ ( \tilde \tau_n^i )  \} $ printed at
the end of the modified algorithm 2 is the law of $ Z^i_{| [ t, t+ t_{max}] } ,$ starting from the fixed past $x$ before time $0.$  The output of the modified algorithm  equals the output
of the unmodified perfect simulation algorithm 2 if $ T^{(i, T )} \leq T . $

We now give a formal argument proving that $\mu$ is indeed a stationary distribution of the process. Let $ f: \mathcal{X}_\infty \to \mathds{R}_+ $ be a bounded measurable function which is cylindrical on $ \{ i \} \times [t, t + t_{max}].$  
Then 
\begin{eqnarray}
 \mathds{E}[ f (Z^i_{| [ t, t+ tmax] } )| Z_{ | \mathds{R}_-} = x  ] &= &\mathds{E} [ f ( \{ ( \tilde \tau_n^i )  \} ), T^{(i, T )}\leq T ]
+ \mathds{E} [ f ( \{ ( \tilde \tau_n^i )  \} ),  T^{(i, T )} > T ] \nonumber \\
& =  &\mathds{E} [ f ( \{ (  \tau_n^i )  \} ), T^{(i, T )} \leq T  ] 
 + \mathds{E} [ f ( \{ ( \tilde \tau_n^i )  \} ),  T^{(i, T )} >  T ] ,
\end{eqnarray}
where $ (  \tau_n^i )$ is the output of the original Perfect simulation algorithm. 

But 
$$ \mathds{E} [ f ( \{ ( \tilde \tau_n^i )  \} ),  T^{(i, T )} >T    ]  \le \| f \|_{\infty} \mathds{P} ( T^{(i, T)} \geq t) \to 0 \mbox{ as } t \to \infty ,$$ 
since finiteness of the tree implies the finiteness of $T^{(i, T)} $ and since by shift invariance, the law of $T^{(i, T)} $ does not depend on $T.$ Hence we obtain that 
$$ \lim_{t \to \infty}  \mathds{E} [ f (Z^i_{| [ t, t+ t_{max}] } )| Z_{ | \mathds{R}_-} = x  ] = \mathds{E} [ f (Z^i_{| [ t, t+ t_{max}] } ) ] ,$$
since $ {\bf 1}_{ T^{(i, T )} \leq  T }   \to 1 $ almost surely. 

This implies that $\mu $ is a stationary distribution of the process. Replacing the initial condition $x $ by any stationary initial condition, chosen within $ \mathcal{Y},$ we finally get also uniqueness of the stationary distribution.} \end{proof}

\subsubsection{The complexity of the algorithm}
\change{Throughout this section we suppose once more that we are in the situation of Item 2. of Proposition \ref{prop:article-2/MethodToWriteKal} and that the subspace $\mathcal{Y} $ is invariant under the dynamics, that is, $ Z_t \in \mathcal{Y} $ implies  $Z_{t+s} \in \mathcal{Y} $ for all $ s \geq 0.$ 
Our goal is to study the effect of the choice of the $(\lambda_i(v))_{v \in \bf V^i } $ on the number of points simulated by our algorithm. Until the end of this section, we suppose moreover that }
\begin{assumption} \label{finite indices}~\\
1. The index set $\textbf{I} = \{ 1, \ldots , N \}$ is finite.\\
2. The sub-criticality assumption \eqref{nbenfant} is satisfied.
\end{assumption}
Let us fix several notations that will be useful in the sequel. We denote $e_i$ the $i$-th unit vector of $\mathds{R}^N$, $\textbf{1}$ is the vector $(1, 1, \ldots, 1)^T$ and $\mu$ still stands for the Lebesgue measure. Finally, by a positive vector, we mean that all its components are positive. 

\change{In the sequel we rely on the multitype branching process $\tilde{\mathcal{T}}$ introduced in the beginning of Section \ref{sec:end}  which is a space-time valued process starting from the ancestor $ (i, T ) $ in generation $ 0. $ In the sequel we refer to these points as ``particles'' and we say  the type of a particle is its associated index value $i.$ Recall that in the definition of this branching process, each particle $(j,t )$ belonging to generation $n,$ independently of anything else, gives rise to offspring particles which are chosen as independent copies of 
$$\mathcal{C}_1^{(j, t ) }  = \bigcup\limits_{k \in \textbf{I}} \left\{(k,s) :  s \in  p_k((V^j_t)^{\rightarrow t})  \mbox{ is a jump of a Poisson process of intensity }\Gamma^k  \right\}.$$
For $n \geq 1$, let $K^i (n)$ be the $N-$dimensional vector containing the numbers of offspring particles of different types belonging to the $n^{th}$ generation of the process issued from  $(i,T),$ that is, $ K^i (n) = ( K^i_1 (n), \ldots, K^i_N (n) )^T , $ where $K^i_k (n) $ is the number of particles of type $k $ within generation $n.$} We use the convention that $K^i (0)  = e_i$ for every $1 \le i \le N .$
\change{For every $j \in \{1, \ldots , N\},$ let
\begin{equation}\label{eq:xij}
 X^i_j := K^i_j(1)
\end{equation}
be the number of offspring particles of type $j$ of the initial particle $(i, T).$  } We have already seen that $X^i_j$ is the cardinal of the points that a Poisson process of intensity $\Gamma^j$ produces on $p_j(V^i_T),$ where $V^i_T$ is the random neighborhood chosen for particle $(i,T).$ In other words, if we denote $\mathcal{P}$ the Poisson distribution, given that $V^i_T= v$, we have
\begin{equation}\label{eq:condlaw}
\change{\mathcal{L}( X^i_j |V^i_T = v )}  =  \mathcal{P}\left(\Gamma^j \mu(p_j(v)) \right). 
\end{equation}

Denoting $X^i = K^1 (1) \in \mathds{R}^N $ the associated vector, we consider for any $\theta \in \mathds{R}^{N}$ the log-Laplace transform of $X^i$ given by 
\begin{equation*}
\phi_i(\theta) := \log \mathds{E}_{i} \left( e^{\theta^{T} X^i} \right)
\end{equation*}
where $\mathds{P}_{i}$ denotes the law of the branching process starting from a single ancestor having type  $i$ and $\mathds{E}_{i}$ is the corresponding expectation. 

%
%
%

Moreover, we consider
\begin{equation*}
W^i (n)  = \sum_{k =0}^{n} K^i(k) 
\end{equation*}
the total number of offspring particles within the first $n$ generations. The log Laplace transform associated to the random vector $W^i(n)$ is given by 
\begin{equation*}
\Phi^{(n)}_{i}(\theta) := \log \mathds{E}_{i}\left(e^{\theta^{T} W^i(n)} \right) , \mbox{ and we put  } \Phi^{(n)}(\theta) = (\Phi^{(n)}_{1}(\theta), \ldots, \Phi^{(n)}_{N}(\theta))^{T}.
\end{equation*} 
Since we are working under the sub-criticality condition \eqref{nbenfant},
\begin{equation}\label{eq:wiinfty}
W^i(\infty) := \lim_{n \to \infty} W^i(n) \mbox{ and  }W^i = \sum_{j=1}^N W^i_j (\infty),
\end{equation}
are well-defined and almost surely finite. In the above formula, $ W^i_j ( \infty )$ denotes the $j-$th coordinate of the $N-$dimensional vector $ W^i ( \infty) , $ that is, the total number of offspring of type $j$ issued from one ancestor particle of type $i.$. In particular, $W^i$ is the total number of offspring particles issued from one ancestor particle of type $i.$ We introduce the associated  log-Laplace transforms 
$$
\Phi_i(\theta):= \log \mathds{E}_{i}\left(e^{{\theta}^{T} W^i(\infty)} \right) , \; \Phi(\theta) = (\Phi_1(\theta), \Phi_2(\theta), \ldots, \Phi_N(\theta))^{T}.
$$

We are now going to state an exponential inequality, inspired by Lemma 1 of \cite{RRB}. \change{To do so, let  $\|.\|_\infty$ be the $\infty-$ norm on $\mathds{R}^N, $ that is, for $ x = (x_1, \ldots, x_N), \| x\|_\infty = \max \{ |x_i |, 1 \le i \le N\}, $} and denote the open balls having center $ 0 \in \mathds{R}^N$ and radius $r$ with respect to this norm by $B(0, r)$. We have to introduce an additional assumption stating that the log-Laplace transform $ \phi_i$ is finite within an open ball around $ 0 .$ 
\begin{assumption}\label{ass:finiteloglap}
There exists $R> 0 $ such that for all positive vectors $\theta$ belonging to  $B(0,R)$ we have
$$
\sup_i \phi_i(\theta) = \sup_{i} \sum_{j=1}^N  \log \left( \sum_{v \in \bf{V}^i } \lambda_i(v ) \exp{ \left[ (e^{\theta_j}- 1) \Gamma^j \mu\left(p_j(v) \right) \right]} \right) < \infty.
$$
\end{assumption}


\begin{proposition} \label{total ancestors} Grant Assumptions \ref{finite indices} and \ref{ass:finiteloglap}. Then the following holds. 
\begin{enumerate}
\item
There exists $ r \in (0, R ) $ such that for all $\theta \in B(0,r  )$, $ \Phi_i(\theta) < \infty$ and moreover 
$$
\Phi(\theta) =  \theta + \phi(\Phi(\theta)).
$$
\item
In particular, for $0 \le \vartheta < r ,$  there exists a constant $c_0$ that depends on $M$ and $i$ such that the following deviation inequality holds for the total number of offspring particles 
\begin{equation} \label{ConIneq}
\mathds{P}_i \left( W^i > \mathds{E}_i(W^i) + x \right) \leq c_0 e^{- \vartheta x} 
\end{equation}
for all $x >0$.
\end{enumerate}
\end{proposition}

\begin{remark}
Let $M$ be  the Jacobian matrix of $\phi(.)$ at $0.$  
Since
\begin{align*}
\phi_i(\theta) = \log \mathds{E}_{i} \left( e^{\theta^{T} X^i} \right)  = \log \mathds{E}_{i} \left( \prod_{j =1}^{N}e^{\theta_j X^i_j} \right),
\end{align*}
we deduce that, using \eqref{eq:condlaw}, 
\begin{equation} \label{MatrixM}
M_{ij}= \mathds{E}_i(X^i_j)= \sum_{v \in \bf{V}^i } \Gamma^j \mu\left(p_j(v) \right) \lambda_i(v) ,
\end{equation}
which is the mean number of offspring particles of type $j, $ issued by a particle of type $i.$ 
\change{It is a well-known fact in branching processes (see e.g. Chapter V of \cite{AN}) that  
$$
\mathds{E}_i \left( (K^i(n))^T  \right) = e_i^T  M^n ,
$$
and that
$$
\mathds{E}_i \left((W^i(n))^T  \right) = \mathds{E}_i \left( \sum_{k = 0}^{n} (K^i(k))^T \right) = e_i^T  \left( \sum_{k =0}^n  M^k \right).
$$
Thus by monotone convergence, }
\begin{equation}\label{impcomplex}
\mathds{E}_i(W^i) =e_i^T  \left( \sum_{k = 0 }^{\infty} M^k \right)  \textbf{1} ,
\end{equation}
which is finite by \eqref{nbenfant}.
\change{Hence \eqref{ConIneq} tells us that the number of points produced in the Backward steps is well concentrated around its mean. If we think that the overall complexity of the algorithm is governed by the number of points produced in the Backward steps, then \eqref{impcomplex} gives us a good way to determine which distribution $\lambda^i$ can lead to  the less costly algorithm. We are looking basically at $\lambda^i$'s that are minimizing \eqref{impcomplex}, with $M$ defined by  \eqref{MatrixM}.}
\end{remark}

\begin{proof}[Proof of Proposition \ref{total ancestors}]~\\

\noindent{\bf Step 1.}
First, we prove that $\phi_i(\theta)$ is well defined for all $\theta \in B(0,r)$. Indeed, since the $(X^i_j)_{j = 1, \ldots, N}$ are independent, we have
$$
\phi_i(\theta) = \log \mathds{E}_i \left( \prod_{j=1}^N \exp{ (\theta_j X^i_j )} \right) 
               = \sum_{j=1}^N \log \mathds{E}_i \left( \exp{ (\theta_j X^i_j) } \right).
$$
By \eqref{eq:condlaw}, 
$$
\mathds{E}_i \left( \exp{ (\theta_j X^i_j) } | V^i_T= v  \right) = \exp{ \left[ (e^{\theta_j}  - 1) \Gamma^j \mu\left(p_j(v) \right) \right]},
$$
and integrating with respect to the choice of $ V^i_T, $ we obtain 
$$
\phi_i(\theta)  =  \sum_{j=1}^N  \log \left( \sum_{v \in \bf {V}^i } \lambda_i(v) \exp{ \left[ (e^{\theta_j} - 1) \Gamma^j \mu\left(p_j(v) \right) \right]} \right) < \infty.
$$
{\bf Step 2.} 
In the following, we prove that $\Phi^{(n)}_{i}(\theta)$ satisfies the following recursion
\begin{equation}\label{eq:recursion}
\Phi^{(n)}_{i}(\theta)  = \theta^{T} K^i(0) + \phi_{i}(\Phi^{(n-1)}(\theta)) .
\end{equation}
Indeed, by definition of $W^i(n)$ we have
$$
\mathds{E}_{i}\left(e^{\theta^{T} W^i(n)} \right) = e^{\theta^{T} K^i(0)} \mathds{E}_{i} \left(e^{\theta^{T} \sum_{k=1}^{n} K^i(k)} \right).
$$
Now, let us introduce for any $ j $ and any $ 1 \le p \le X^i_j , $ $K_{(p)}^j (n-1)$ which is the vector of offspring particles within the $(n-1)^{th}$ generation,  issued from the $p$-th particle  of type $j$ in the first generation. Notice that by the branching property, for $p = 1, \ldots , X^i_j$, we have that the $K_{(p)}^j(n-1)$'s are independent copies of $K^{j}(n-1). $ Therefore, conditioning on the first generation offspring particles, 
\begin{multline*}
\mathds{E}_{i} \left(e^{\theta^{T} \sum_{k=1}^{n} K^i(k)} \right)  = \mathds{E}_{i} \left(e^{\theta^{T} \sum_{k=1}^{n} \sum_{j =1}^{N} \sum_{p =1}^{X^i_j}K_{(p)}^j(k-1)} \right) \\
= \mathds{E}_{i} \left(\prod_{j =1}^{N}  \prod_{p=1}^{X^i_j}  e^{\theta^{T} \sum_{k=1}^{n}  K_{(p)}^j(k-1)} \right) 
=\mathds{E}_{i} \left[ \mathds{E} \left( \prod_{j =1}^{N} \prod_{p=1}^{X^i_j} e^{\theta^{T} \sum_{k=1}^{n}   K_{(p)}^j(k-1)} | X^i \right)  \right].
\end{multline*}
\change{Since the $(K_{(p)}^j(k-1))_{1 \le j \le N} $ are independent and independent of $ X^i ,$ we obtain }
\begin{multline*}
\mathds{E}_{i} \left[ \mathds{E} \left( \prod_{j =1}^{N} \prod_{p=1}^{X^i_j} e^{\theta^{T} \sum_{k=1}^{n}   K_{(p)}^j(k-1)} | X^i \right)  \right]  
= \mathds{E}_{i} \left[  \prod_{j =1}^{N} \prod_{p=1}^{X^i_j} \mathds{E} \left( e^{\theta^{T} \sum_{k=1}^{n}K_{(p)}^j(k-1)} \mid X^i_j \right) \right] \\
=  \mathds{E}_{i} \left[  \prod_{j =1}^{N} \prod_{p=1}^{X^i_j} \mathds{E} \left( e^{\theta^{T} \sum_{k=1}^{n}K_{(p)}^j(k-1)} \right) \right] 
= \mathds{E}_{i} \left[  \prod_{j =1}^{N}  \left( \mathds{E} \left( e^{\theta^{T} W^j(n-1)} \right) \right)^{X^i_j} \right] \\
= \mathds{E}_{i} \left[  \prod_{j =1}^{N}  e^{\Phi^{(n-1)}_{j}(\theta) X^i_j} \right] =  \mathds{E}_{i} \left[ e^{\Phi^{(n-1)}(\theta)^{T} X_i} \right] = e^{\phi_{i}(\Phi^{(n-1)}(\theta))},
\end{multline*}
implying
\begin{equation*}
\Phi^{(n)}_{i}(\theta)  = \theta^{T} K^i(0) + \phi_{i}(\Phi^{(n-1)}(\theta)) ,
\end{equation*}
or, in vector form, 
\begin{equation} \label{eq6.1}
\Phi^{(n)}(\theta) = \theta+ \phi(\Phi^{(n-1)}(\theta)) .
\end{equation}
{\bf Step 3.} \change{Let us consider the sums of the elements of the $i-$th line of the  matrix $M,$
$$
\sum_{j=1}^N   M_{ij} = \sum_{j=1}^N  \mathds{E}_{i} (X^i_j) .$$
By definition, this is the mean number of offspring particles (of any type), issued from  a particle of type $i, $ and by the arguments presented in the proof of Proposition \ref{prop:article-2/Stopping condition}, this mean number is given by 
$$\sum_{j=1}^N   M_{ij}  = \sum_{v \in \bf {V}^i } \lambda^i(v) P(v) \le \gamma < 1,
$$
where the last upper bound holds by our assumptions. }Hence, $\|M\|_{\infty} = \sup_{\|x\|_\infty \leq 1} \{ \|Mx\|_\infty \}=  \sup_i \sum_{j =1}^N |M_{ij }|  \le \gamma < 1 ,$ where $\|.\|_\infty$ is the induced norm for matrices on $\mathds{R}^{N \times N}$. Therefore,  the derivative of $\phi$ in $0$ satisfies, $\|D\phi(0)\|_\infty \leq \gamma < 1$. Moreover, since the norm is continuous and $D\phi(s)$ is likewise, for any $ \gamma < C < 1,$ there is a $\tilde r \in (0, R), $ such that, for $||s||_\infty < \tilde r ,$
$$
\|D\phi(s)\|_\infty \leq C <1 .
$$
Hence, $\phi(s)$ is Lipschitz continuous in the ball $B(0,\tilde r)$ and moreover $\phi(0)=0$, which implies that
$$
\| \phi(s) \|_\infty \leq C \| s\|_\infty 
$$
for $\|s\|_\infty < \tilde r$.

Now, take  $\theta$ such that $| \theta_i | \leq \tilde r ( 1- C) $ for any $ 1 \le i \le N .$ 
\change{Using \eqref{eq6.1}, we can show by induction} that for all $n, $ 
\begin{equation*}
\|\Phi^{(n)}(\theta)\|_\infty \leq \|\theta\|_\infty (1 + C + \ldots + C^n) \leq \tilde r < \infty.
\end{equation*}
\change{Hence} by monotone convergence of $ \Phi^n ( \theta) \to \Phi ( \theta)  $ and  the limit in \eqref{eq6.1},  \change{we have} 
$$
\|\Phi (\theta)\|_\infty \leq r \mbox{ and }  \Phi(\theta) =  \theta + \phi(\Phi(\theta)),
$$
with $r =  \tilde r ( 1 - C).$ 
In particular, choosing $\theta = \vartheta \textbf{1} \in B(0,\change{r})$ with $\vartheta \in \mathds{R}$, we have 
$$
\mathds{E}_{i}\left(e^{\vartheta W^i} \right) < \infty ,
$$
where we recall that  $W^i$ is the total number of offspring particles of $ i.$ \\

\noindent{\bf Step 4.}
\change{
We use Markov's inequality and obtain for any $ \vartheta \in ( 0, r) ,$  
$$ \mathds{P}_i \left( W^i > \mathds{E}_i(W^i) + x \right) =  \mathds{P}_i \left( e^{ \vartheta W^i } > e^{ \vartheta \mathds{E}_i(W^i) + \vartheta x} \right) \le 
e^{ \tilde \Phi_i ( \vartheta) - \vartheta \mathds{E}_i(W^i) -  \vartheta x} , $$ 
where 
$ \tilde \Phi_i(\vartheta )= \log \mathds{E}_{i}\left(e^{{\vartheta} W^i} \right) .$ Using Taylor's formula, we have 
$$  \tilde \Phi_i ( \vartheta) =  \tilde \Phi_i ' (0) \vartheta + \frac12 \tilde \Phi_i^{''} ( \tilde \vartheta) \vartheta^2 = \mathds{E}_i(W^i) \vartheta + \frac12 \tilde \Phi_i^{''} ( \tilde \vartheta) \vartheta^2 ,$$
for some $ \tilde \vartheta \in (0, \vartheta).$ Choosing $c_0 =  \sup_{ \vartheta \le r } e^{ \frac12  \tilde \Phi_i^{''} ( \tilde  \vartheta) \vartheta^2} $ implies the assertion. }
\end{proof}

\subsubsection{Choice of the weights on a particular example}\label{sec:choice}
In this section we discuss a very particular example to show how to calibrate the choice of the $\lambda^i$'s in terms of minimizing the number of points produced int he Backward steps. Our example is the age dependent Hawkes process with hard refractory period $\delta>0$ defined by \eqref{eq: hard refractory period} and with non-decreasing $\psi^i$ which is $L$-Lipschitz for all $ i, $ so that Corollary \ref{Cor:age} Item 2. applies. The following choices are merely directed to have the simplest possible computations on a non trivial infinite case. 

In what follows, to simplify the computations, we consider that $\textbf{I} = \mathds{Z}$ and that $ L=1.$ Moreover, for all $i$, we set
\begin{enumerate}
\item $\psi^i(0)=1.$
\item $h^i_j(t) = \beta^i_j \exp{(- t/\delta )},$ 
\change{where} $\beta^i_j = \dfrac{1}{2 |j-i|^{\gamma}}$ for $j \neq i,$ \change{for some fixed  positive parameter} $\gamma>1,$ and $\beta^i_{i} = 1,$ where $ \beta_j^i \geq 0.$ 

Note that $\|h^i_j\|_1=\beta^i_j  \delta$ and that $h^i_j(0)=\beta^i_j.$
\item \change{With the notation of  Corollary \ref{Cor:age}, we use} $\omega^i_1=\{ i \} $ and $  \omega^i_k = \{ i-k+1, \ldots, i, i+1, \ldots, i+k-1 \}$  for all $ k \geq 2$.
\end{enumerate}
\change{With this choice, one can see that $\bar{\Gamma}^i_k$ defined in Corollary \ref{Cor:age} Item 2. is given for $k\geq 2$ by 
$$\bar{\Gamma}^i_k=\frac{2}{(k-1)^\gamma}+e^{-(k-1)} \left(1+\sum_{m=1}^{k-2} \frac{1}{m^\gamma}\right) .$$
Hence for some constant $C_\gamma>0$ depending on $\gamma , $ 
one can always choose for all $k\geq 1$
$$\Gamma^i_k \geq C_\gamma  k^{-\gamma},$$
as long as $\sum_k\Gamma^i_k<\infty$.
So let us take 
$$\Gamma^i_k =C_\gamma  k^{-p} ,$$
for some $1<p\leq \gamma$. Notice that in this case, $\Gamma^i$ is independent on $i$, since 
$$\Gamma^i=\sum_{k=1}^\infty \Gamma^i_k = C_\gamma c_p , \; \mbox{ where }  c_p = \sum_{k=1}^\infty k^{-p }.$$
By applying Proposition \ref{prop:article-2/MethodToWriteKal} Item 2. combined with Corollary \ref{Cor:age} Item 2, we obtain that 
$$\lambda^i(v^i_k)= (c_{p})^{-1} k^{-p}.$$
Now let us turn to finding $p$ such that the $\lambda^i$'s are minimizing \eqref{impcomplex}, that is 
 $$\mathds{E}_i(W^i) =e_i^T  \left( \sum_{k = 0 }^{\infty} M^k \right)  \textbf{1}$$
with $M$ defined by 
$$
M_{ij}= \mathds{E}_i(X^i_j)= \sum_{k=1}^\infty \Gamma^j \mu\left(p_j(v^i_k) \right) \lambda^i(v^i_k).$$
Summing over all possible types $j, $ we obtain the mean number of offspring particles of a particle of type $i, $ which is given by 
$$ M_i = \sum_{j =1}^\infty M_{ij }   =C_{\gamma} \delta  \sum_{k=1}^\infty   (2k - 1 ) k^{1-p} =: C_{\gamma} \delta f(p), \; \mbox{ where } f(p) =  \sum_{k=1}^\infty   (2k - 1 ) k^{1-p}.
$$
Clearly, $ f (p) < \infty $ if and only if $ p > 3,$ and thus, a fortiori, $\gamma > 3.$ Since $f$ is a decreasing function of $p,$ the optimal choice for $p$ is thus $p = \gamma .$ 
}

\section*{Acknowledgements}
This paper is dedicated to the Institute of Mathematics Hanoi, where \change{T.C. Phi} was honored to work for more than three years. \change{ We also warmly thank the referees of an earlier version of this paper for the very detailed comments that helped us to improve the manuscript}.

This work was supported by the French government, through the UCAJedi and 3IA Côte d’Azur Investissements d’Avenir managed by the National Research Agency (ANR-15- IDEX-01 and ANR- 19-P3IA-0002), by the CNRS through the "Mission pour les Initiatives Transverses et Interdisciplinaires" (Projet DYNAMO, "APP Modélisation du Vivant"), by the interdisciplinary Institute for Modeling in Neuroscience and Cognition (NeuroMod) of the Université Côte d’Azur,  and directly by the  National Research Agency (ANR-19-CE40-0024) with the ChaMaNe project. Moreover, this work is part of  the  FAPESP project Research, Innovation and Dissemination Center for Neuromathematics(grant 2013/07699-0)

\bibliographystyle{plain}
\bibliography{mybib-2} 
\end{document}